\documentclass[a4paper,12pt,reqno]{amsart}
\usepackage{a4wide}
\usepackage{amsmath}
\usepackage{amssymb}
\usepackage{amsthm}
\usepackage{latexsym}
\usepackage{graphicx}
\usepackage[english]{babel}
              
\newtheorem{prop}[subsection]{Proposition}
\newtheorem{conj}[subsection]{Conjecture}
\newtheorem{teor}[subsection]{Theorem}
\newtheorem{lema}[subsection]{Lemma}
\newtheorem{cor} [subsection]{Corollary}
\theoremstyle{definition}

\theoremstyle{remark}
\newtheorem{obs} [subsection]{Remark}

\def\qdepth{\operatorname{hdepth}}
\def\hdepth{\operatorname{hdepth}}

\numberwithin{equation}{section}

\begin{document}

\title[Comparing Hilbert depth of $I$ with Hilbert depth of $S/I$. III]{Comparing Hilbert depth of $I$ with Hilbert depth of $S/I$. III}
\author[Andreea I.\ Bordianu, Mircea Cimpoea\c s 
       ]
  {Andreea I.\ Bordianu$^1$ and Mircea Cimpoea\c s$^2$
	}
\date{}



\footnotetext[1]{ \emph{Andreea I.\ Bordianu}, University Politehnica of Bucharest, Faculty of
Applied Sciences, 
Bucharest, 060042, E-mail: andreea.bordianu@stud.fsa.upb.ro}
\footnotetext[2]{ \emph{Mircea Cimpoea\c s}, University Politehnica of Bucharest, Faculty of
Applied Sciences, 
Bucharest, 060042, Romania and Simion Stoilow Institute of Mathematics, Research unit 5, P.O.Box 1-764,
Bucharest 014700, Romania, E-mail: mircea.cimpoeas@upb.ro,\;mircea.cimpoeas@imar.ro}

\begin{abstract}
Let $I$ be a squarefree monomial ideal of $S=K[x_1,\ldots,x_n]$. We prove that if $\hdepth(S/I)\leq 8$ or $n\leq 10$ 
then $\hdepth(I)\geq \hdepth(S/I)-1$.
\end{abstract}

\maketitle

\section{Introduction}

Let $K$ be a field and $S=K[x_1,\ldots,x_n]$ the polynomial ring over $K$ in $n$ indeterminates.
Let $M$ be a finitely generated graded $S$-module. The Hilbert depth of $M$, denoted by $\hdepth(M)$, is the 
maximal depth of a finitely generated graded $S$-module $N$ with the same Hilbert series as $M$.
For basic properties of this invariant we refer the reader to \cite{bruns,uli}.

Let $0\subset I\subsetneq J\subset S$ be two squarefree monomial ideals.
In \cite{lucrare2}, it was presented a new method of computing the Hilbert depth of $J/I$, as follows:

For all $0\leq j\leq n$, let $\alpha_j(J/I)$ be the number of squarefree monomials 
$u$ of degree $j$ with $u\in J\setminus I$. For all $0\leq q\leq n$ and $0\leq k\leq q$, we consider the integers
\begin{equation}\label{betak}
  \beta_k^q(J/I):=\sum_{j=0}^k (-1)^{k-j} \binom{q-j}{k-j} \alpha_j(J/I).
\end{equation}
In \cite[Theorem 2.4]{lucrare2} it was proved that the Hilbert depth of $J/I$ is
\begin{equation}\label{hdep}
\qdepth(J/I):=\max\{q\;:\;\beta_k^q(J/I) \geq 0\text{ for all }0\leq k\leq q\leq n\}.
\end{equation}
Using this combinatorial characterization of the Hilbert depth and the Kruskal-Katona Theorem, in \cite{bordi} and \cite{bordi2}
we studied connections between the Hilbert depth of $S/I$ and Hilbert depth of $I$, proving that if $\hdepth(S/I)\leq 6$ or $n\leq 9$ then
$\hdepth(I)\geq \hdepth(S/I)$. Moreover, this result is best possible, in the sense that we provided an example with $\hdepth(S/I)=7$
and $\hdepth(S/I)=6$ in $S=K[x_1,\ldots,x_{10}]$; see \cite[Example 3.18]{bordi}.

The aim of our paper is to continue this study, using similar techniques.
We prove that if $I\subset S$ is a squarefree monomial with $\hdepth(S/I)\leq 8$ or $n\leq 10$ then 
$\hdepth(I)\geq \hdepth(S/I)-1$.
Also, we propose a general conjecture regarding the relation between $q:=\hdepth(S/I)$ and $\hdepth(I)$
with respect to the number of variables $n$ and $q$; see Conjecture \ref{conju}.

\section{Main results}

Let $I\subset S=K[x_1,\ldots,x_n]$ be a squarefree monomial ideal with $\qdepth(S/I)=q\leq n-1$.
For convenience, we denote 
$$\alpha_j=\alpha_j(S/I), 0\leq j\leq n,\text{ and }\beta_k^q=\beta_k^q(S/I),\; 0\leq k\leq q\leq n.$$
We recall the following result:

\begin{teor}(see \cite[Theorem 2.2]{bordi})\label{teo1}
The following are equivalent:
\begin{enumerate}
\item[(1)] $I$ is principal.
\item[(2)] $\qdepth(I)=n$.
\item[(3)] $\qdepth(S/I)=n-1$.
\end{enumerate}
\end{teor}

Theorem \ref{teo1} shows that, in order to prove that $\qdepth(I)\geq \qdepth(S/I)$, it is safe to assume that
$I$ is not principal. Hence $n\geq q+2$. Another important reduction which can be done is to assume that $I\subset \mathfrak m^2$;
see \cite[Remark 2.4]{bordi}. Hence $\alpha_0=1$ and $\alpha_1=n$.

We recall also the following combinatorial identity:
\begin{equation}\label{combi}
\sum_{j=0}^k (-1)^{k-j}\binom{q-j}{k-j}\binom{n}{j} = \binom{n-q+k-1}{k},\text{ for all }0\leq q\leq k\leq n,
\end{equation}
which follows immediately from Chu-Vandermonde formula. Since 
$$\alpha_j(S/I)+\alpha_j(I)=\binom{n}{j},\text{ for all }0\leq j\leq n,$$ 
from \eqref{combi} and \eqref{betak} it follows that
\begin{equation}\label{betai}
\beta_k^q(I) = \beta_k^q(S/I) - \binom{n-q+k-1}{k},\text{ for all }0\leq q\leq k\leq n.
\end{equation}
Let $N$ and $k$ be two positive integers. Then $N$ can be uniquely written as 
$$N=\binom{n_k}{k}+\binom{n_{k-1}}{n_{k-1}}+\cdots+\binom{n_j}{j},\text{ where }n_k>n_{k-1}>\cdots>n_j\geq j\geq 1.$$
As a direct consequence of the Kruskal-Katona Theorem (\cite[Theorem 2.1]{stanley}), we have the following restrictions on $\alpha_j$'s:

\begin{lema}(see \cite[Lemma 3.2]{bordi})\label{cord}
If $\alpha_k=\binom{n_k}{k} + \binom{n_{k-1}}{k-1}+\cdots+\binom{n_j}{j}$, as above, where $2\leq k\leq n-1$, then
\begin{enumerate}
\item[(1)] $\alpha_{k-1}\geq \binom{n_k}{k-1}+\binom{n_{k-1}}{k-2}+\cdots+\binom{n_j}{j-1}$.
\item[(2)] $\alpha_{k+1}\leq \binom{n_k}{k+1}+\binom{n_{k-1}}{k}+\cdots+\binom{n_j}{j+1}$.
\end{enumerate}
\end{lema}


\begin{lema}\label{b37}
Let $I\subset S$ be a proper squarefree monomial ideal with $q=\qdepth(S/I)\leq 8$. Then: 
$$\beta_3^{q-1}(S/I)\leq \binom{n-q+3}{3}.$$
\end{lema}

\begin{proof}
Let $q'=q-1$. Note that, in the proof of \cite[Lemma 4.1]{bordi}, we used only the fact that $\hdepth(S/I)\geq q$.
Therefore, as $\hdepth(S/I)\geq q'$, according to \cite[Lemma 4.1]{bordi}, we have $\beta_3^{q'}(S/I)\leq \binom{n-q'+2}{3}$.
Hence, we are done.
\end{proof}

\begin{lema}\label{b46}
Let $I\subset S$ be a proper squarefree monomial ideal with $5\leq q=\qdepth(S/I)\leq 7$. Then: 
$$\beta_4^{q-1}\leq \binom{n-q+4}{4}.$$
\end{lema}

\begin{proof}
The conclusion follows from \cite[Theorem 3.15]{bordi}, \cite[Lemma 4.2]{bordi} and \cite[Lemma 2.5]{bordi2},
using a similar argument as in the proof of the previous lemma.
\end{proof}

\section{Main results}

In the following, we will assume that $I\subset \mathfrak m^2$ is a squarefree monomial 
ideal which is not principal. Let $q=\qdepth(S/I)$. Since $I$ is not principal, it follows that 
$n\geq q+2$. On the other hand, since $I\subset \mathfrak m^2$, we have $\hdepth(I)\geq 2$.

We first prove the following lemma:

\begin{lema}\label{lem2}
Let $I\subset S$ be a proper squarefree monomial ideal with $\qdepth(S/I)=q\geq 3$ and let $0\leq \ell \leq q-2$ be an integer.
The following are equivalent:
\begin{enumerate}
\item[(1)] $\qdepth(I)\geq \qdepth(S/I)-\ell.$
\item[(2)] $\beta_{k-\ell}^{q-\ell}(S/I) \leq \binom{n-q+k-1}{k-\ell},\text{ for all }\ell+3\leq k\leq q$.
\end{enumerate}
\end{lema}

\begin{proof}
Since, from \eqref{betai}, we have 
$$\beta_{k-\ell}^{q-\ell}(S/I) + \beta_{k-\ell}^{q-\ell}(I) = \binom{n-q+k-1}{k-\ell},$$ 
it follows that $\qdepth(I)\geq \qdepth(S/I)-\ell$ if and only if 
\begin{equation}\label{conditz}
\beta_{k-\ell}^{q-\ell}(S/I) \leq \binom{n-q+k-1}{k-\ell},\text{ for all }\ell\leq k\leq q. 
\end{equation}
On the other hand, by our assumptions, we have:
$$\beta_0^{q-\ell}(I)=\beta_1^{q-\ell}(I)=0\text{ and }\beta_2^{q-\ell}(I)=\alpha_2(I)\geq 0,$$
hence \eqref{conditz} is automatically fulfilled for $k\leq \ell+2$.
\end{proof}

\begin{lema}\label{q3_10}
If $q\leq 10$, then $\beta_3^{q-1} \leq \binom{n-q+3}{3}$.
\end{lema}

\begin{proof}
According to Lemma \ref{b37}, we can assume that $q\in \{9,10\}$.
Since $$\beta_2^q = \alpha_2 - \binom{q-1}{1}n + \binom{q}{2} \geq 0,$$ it follows that 
\begin{equation}\label{condx}
\alpha_2 \geq (q-1)n - \frac{q(q-1)}{2} = \frac{1}{2}(q-1)(2n-q).
\end{equation}
On the other hand, since 
\begin{align*}
& \beta_3^{q-1} = \alpha_3 - (q-3)\alpha_2 + \binom{q-2}{2}n - \binom{q-1}{3}\text{ and }\\
& \binom{n-q+3}{3} = \binom{n}{3} - (q-3)\binom{n}{2} + \binom{q-2}{2}n - \binom{q-1}{3},
\end{align*}
 in order to
complete the proof, it is enough to show that
\begin{equation}\label{vrere}
\alpha_3 - (q-3)\alpha_2 \leq \binom{n}{3} - (q-3)\binom{n}{2}.
\end{equation}
If $\alpha_2=\binom{n}{2}$ then there is nothing to prove, since $\alpha_3\leq \binom{n}{3}$, so we may
assume that $\alpha_2=\binom{n_2}{2}+\binom{n_1}{1}$ where $n>n_2>n_1\geq 0$. From Lemma \ref{cord} it follows
that $\alpha_3\leq \binom{n_2}{3}+\binom{n_1}{2}$.
We consider the functions $$f(x)=\binom{x}{3}-(q-3)\binom{x}{2}\text{ and }g(x)=\binom{x}{2}-(q-3)\binom{x}{1}.$$
Note that, from the above, we have $$\alpha_3-(q-3)\alpha_2\leq f(n_2)+g(n_1).$$ Hence, in order to prove \eqref{vrere},
by \eqref{condx}, it is enough to show that
\begin{equation}\label{pohta}
f(n_2)+g(n_1) \leq f(n),\text{ if }\alpha_2 \geq \frac{1}{2}(q-1)(2n-q).
\end{equation}
We have two cases to consider:
\begin{enumerate}
\item[(1)] $q=9$. We have the following table of values for $f(x)$ and $g(x)$:

\begin{center}
\begin{table}[tbh]
\begin{tabular}{|c|c|c|c|c|c|c|c|c|c|c|c|}
\hline
$x$   & 0 & 1 & 2  & 3  & 4  & 5  & 6  & 7  & 8   & 9   & 10  \\ \hline
$f(x)$ & 0 & 0 & -6 & -17& -32& -50& -70& -91& -112& -132& -150 \\ \hline
$g(x)$ & 0 & -6& -11& -15& -18& -20& -21& -21& -20 & -18 & -15 \\ \hline
\end{tabular}
\end{table}

\begin{table}[tbh]
\begin{tabular}{|c|c|c|c|c|c|c|c|c|c|c|c|}
\hline
$x$    &  11 & 12  & 13  & 14  & 15  & 16  & 17  & 18  & 19 & 20 \\ \hline
$f(x)$ & -165& -176& -182& -182& -175& -160& -136& -102& -57& 0   \\ \hline
$g(x)$ &  -11& -6  & 0    & 7  & 15  & 24  & 34  & 45  & 57 & 70   \\ \hline
\end{tabular}
\end{table}
\end{center}

Note that $f(x)$ is increasing on $[14,\infty)$ and $g(x)$ is increasing on $[7,\infty)$.
Also, $f(x)\leq 0$ for $x\leq 20$ and $f(x)\geq 0$ for $x\geq 20$. Thus, for $n\geq 21$ 
we have that
$$ f(n)-(f(n_2)+g(n_1))\geq f(n)-(f(n-1)+g(n-2))= n-8 > 0,$$
and thus we are done by \eqref{pohta}. Hence, we may assume $11\leq n\leq 20$.
\begin{itemize}
\item $n=20$. From \eqref{condx} we have $\alpha_2\geq 124=\binom{16}{2}+\binom{4}{1}$. It follows that
              $$f(n_2)+g(n_1)\leq f(19)+g(18)=-12<f(20)=0,$$
							and thus we are done  by \eqref{pohta}.
\item $n=19$. From \eqref{condx} we have $\alpha_2\geq 116=\binom{15}{2}+\binom{11}{1}$.
              It follows that
              $$f(n_2)+g(n_1)\leq f(18)+g(17)=-78<f(19)=-57,$$
							and thus we are done  by \eqref{pohta}.
\item $n=18$. From \eqref{condx} we have $\alpha_2\geq 108=\binom{15}{2}+\binom{3}{1}$.
              It follows that
              $$f(n_2)+g(n_1)\leq f(17)+g(16)=-112<f(18)=-102,$$
							and thus we are done  by \eqref{pohta}.
\item $n=17$. From \eqref{condx} we have $\alpha_2\geq 100=\binom{14}{2}+\binom{9}{1}$.
							It follows that
              $$f(n_2)+g(n_1)\leq f(16)+g(15)=-145<f(17)=-136,$$
							and thus we are done  by \eqref{pohta}.
\item $n=16$. From \eqref{condx} we have $\alpha_2\geq 92=\binom{14}{2}+\binom{1}{1}$.
              It follows that
              $$f(n_2)+g(n_1)\leq f(15)+g(14)=-168<f(16)=-160,$$
							and thus we are done  by \eqref{pohta}.
\item $n=15$. From \eqref{condx} we have $\alpha_2\geq 84=\binom{13}{2}+\binom{6}{1}$.
              It follows that
              $$f(n_2)+g(n_1)\leq f(14)=-182<f(15)=-175,$$
							and thus we are done  by \eqref{pohta}.
\item $n=14$. From \eqref{condx} we have $\alpha_2\geq 76=\binom{12}{2}+\binom{10}{1}$.
              It follows that
              $$f(n_2)+g(n_1)\leq f(12)+g(11)=-187<f(16)=-182,$$
							and thus we are done  by \eqref{pohta}.
\item $n=13$. From \eqref{condx} we have $\alpha_2\geq 68=\binom{12}{2}+\binom{2}{1}$.
              It follows that
              $$f(n_2)+g(n_1)\leq f(12)+g(2)=-187<f(13)=-182,$$
							and thus we are done  by \eqref{pohta}.
\item $n=12$. From \eqref{condx} we have $\alpha_2\geq 60=\binom{11}{2}+\binom{5}{1}$.
              It follows that
              $$f(n_2)+g(n_1)\leq f(11)+g(10)=-180<f(12)=-176,$$
							and thus we are done  by \eqref{pohta}.
\item $n=11$. From \eqref{condx} we have $\alpha_2\geq 52=\binom{10}{2}+\binom{7}{1}$.
              It follows that
              $$f(n_2)+g(n_1)\leq f(10)+g(9)=-168<f(11)=-165,$$
							and thus we are done  by \eqref{pohta}.
\end{itemize}

\pagebreak

\item[(2)] $q=10$. We have the following table of values for $f(x)$ and $g(x)$:

\begin{center}
\begin{table}[tbh]
\begin{tabular}{|c|c|c|c|c|c|c|c|c|c|c|c|c|}
\hline
$x$    & 1 & 2  & 3  & 4  & 5  & 6  & 7   & 8   & 9   & 10  & 11   \\ \hline
$f(x)$ & 0 & -7 & -20& -38& -60& -85& -112& -140& -168& -195& -220 \\ \hline
$g(x)$ & -7& -13& -18& -22& -25& -27& -28 & -28 & -27 & -25 & -22 \\ \hline
\end{tabular}
\end{table}

\begin{table}[tbh]
\begin{tabular}{|c|c|c|c|c|c|c|c|c|c|c|c|c|}
\hline
$x$    &  12 & 13  & 14  & 15  & 16  & 17  & 18  & 19  & 20  & 21  & 22 & 23 \\ \hline
$f(x)$ & -242  & -260& -273& -280& -280& -272& -255& -228& -190& -140& -77& 0 \\ \hline
$g(x)$ & -18  & -13 & -7  & 0   & 8   & 17  & 27  & 38  & 50  & 63  & 77 & 92 \\ \hline
\end{tabular}
\end{table}
\end{center}

Note that $f(x)$ is increasing on $[16,\infty)$ and $g(x)$ is increasing on $[8,\infty)$.
Also, $f(x)\leq 0$ for $x\leq 23$ and $f(x)\geq 0$ for $x\geq 23$. Thus, for $n\geq 24$ 
we have that:
$$ f(n)-(f(n_2)+g(n_2))\geq f(n)-(f(n-1)+g(n-2))= n-9 > 0,$$
and therefore we are done by \eqref{pohta}. Hence, we may assume $12\leq n\leq 23$.
\begin{itemize}
\item $n=23$. From \eqref{condx} we have $\alpha_2\geq 162=\binom{18}{2}+\binom{9}{1}$.
							It follows that
              $$f(n_2)+g(n_1)\leq f(22)+g(21)=-14<f(23)=0,$$
							and thus we are done  by \eqref{pohta}.
\item $n=22$. From \eqref{condx} we have $\alpha_2\geq 153=\binom{18}{2}$.
							It follows that
              $$f(n_2)+g(n_1)\leq f(21)+g(20)=-90<f(22)=-77,$$
							and thus we are done  by \eqref{pohta}.
\item $n=21$. From \eqref{condx} we have $\alpha_2\geq 144=\binom{17}{2}+\binom{8}{1}$.
							It follows that
              $$f(n_2)+g(n_1)\leq f(20)+g(19)=-152<f(21)=-140,$$
							and thus we are done  by \eqref{pohta}.
\item $n=20$. From \eqref{condx} we have $\alpha_2\geq 135=\binom{16}{2}+\binom{15}{1}$.
							It follows that
              $$f(n_2)+g(n_1)\leq f(19)+g(18)=-201<f(20)=-190,$$
							and thus we are done  by \eqref{pohta}.
\item $n=19$. From \eqref{condx} we have $\alpha_2\geq 126=\binom{16}{2}+\binom{6}{1}$.
							It follows that
              $$f(n_2)+g(n_1)\leq f(18)+g(17)=-238<f(19)=-228,$$
							and thus we are done  by \eqref{pohta}.
\item $n=18$. From \eqref{condx} we have $\alpha_2\geq 117=\binom{15}{2}+\binom{12}{1}$.
							It follows that
              $$f(n_2)+g(n_1)\leq f(17)+g(16)=-264<f(18)=-255,$$
							and thus we are done  by \eqref{pohta}.
\item $n=17$. From \eqref{condx} we have $\alpha_2\geq 108=\binom{15}{2}+\binom{3}{1}$.
							It follows that
              $$f(n_2)+g(n_1)\leq f(16)=-280<f(17)=-272,$$
							and thus we are done  by \eqref{pohta}.
\item $n=16$. From \eqref{condx} we have $\alpha_2\geq 99=\binom{14}{2}+\binom{8}{1}$.
							It follows that
              $$f(n_2)+g(n_1)\leq f(15)=f(16)=-280,$$
							and thus we are done  by \eqref{pohta}.
\item $n=15$. From \eqref{condx} we have $\alpha_2\geq 90=\binom{13}{2}+\binom{12}{1}$.             
							If $\alpha_2\geq 92=\binom{14}{2}+\binom{1}{1}$, it follows that
              $$f(n_2)+g(n_1)\leq f(14)+g(1)=-280=f(15),$$
							and thus we are done  by \eqref{pohta}. Hence we may assume that $\alpha_2\in \{90,91\}$.
							
							\vspace{5pt}
							If $\alpha_2=91=\binom{14}{2}$ then $\alpha_3\leq \binom{14}{3}=364$ and $\alpha_4\leq \binom{14}{4}=1001$.
							\vspace{5pt}
							
							If $\alpha_3\leq 357$ then 
							$$\alpha_3-7\alpha_2 \leq 357 - 7\cdot 91 = -280=f(15),$$
							and we are done. So we may assume that $\alpha_3\geq 358$. Since
							$$\beta_4^{10} = \alpha_4 - 7\alpha_3 + \binom{8}{2}\alpha_2 -\binom{9}{3}n+\binom{10}{4} = \alpha_4 - 7\alpha_3 + 1498\geq 0,$$
							it follows that $\alpha_4\geq 7\cdot 358 - 1498 = 1008$, a contradiction.							
							
							\vspace{5pt}
							If $\alpha_2=90$ then $\alpha_3\leq \binom{13}{3}+\binom{12}{2}=352$ and $\alpha_4\leq \binom{13}{4}+\binom{12}{3}=935$.
							\vspace{5pt}
							
							If $\alpha_3\leq 350$ then $$\alpha_3-7\alpha_2 \leq 350 - 7\cdot 90 = -280=f(15),$$
							and we are done again. So we may assume $\alpha_3\geq 351$. Since 
							$$\beta_4^{10} =\alpha_4 - 7\alpha_3 + \binom{8}{2}\alpha_2 -\binom{9}{3}n+\binom{10}{4} = \alpha_4 - 7\alpha_3 + 1470 \geq 0,$$ 						
							we get $\alpha_4\geq 7\cdot 351 - 1470 = 987$, 
							a contradiction.
														
\item $n=14$. From \eqref{condx} we have $\alpha_2\geq 81=\binom{13}{2}+\binom{3}{1}$.
							It follows that
              $$f(n_2)+g(n_1)\leq f(13)+g(3)=-273=f(14),$$
							and thus we are done  by \eqref{pohta}.
\item $n=13$. From \eqref{condx} we have $\alpha_2\geq 72=\binom{12}{2}+\binom{6}{1}$.
							It follows that
              $$f(n_2)+g(n_1)\leq f(12)+g(11)=-264<f(13)=-260,$$
							and thus we are done  by \eqref{pohta}.
\item $n=12$. From \eqref{condx} we have $\alpha_2\geq 63=\binom{11}{2}+\binom{8}{1}$.
							It follows that
              $$f(n_2)+g(n_1)\leq f(11)+g(10)=-245<f(12)=-242,$$
							and thus we are done  by \eqref{pohta}.
\end{itemize}
\end{enumerate}
Hence, the proof is complete.
\end{proof}

\pagebreak

\begin{lema}\label{q4_8}
If $q=8$ then $\beta_4^{7} \leq \binom{n-4}{4}$.
\end{lema}

\begin{proof}
Since $\beta_2^8\geq 0$ and $\beta_3^8\geq 0$ it follows that 
\begin{equation}\label{condi8}
\alpha_2 \geq 7n-28,\;\alpha_3 \geq 6\alpha_2-21n+56 \geq 21n-112.
\end{equation}
Since $\beta_4^{7}=\alpha_4-4\alpha_3+10\alpha_2-20n+35$ and $\alpha_2\leq \binom{n}{2}$, 
in order to prove that $\beta_4^7\leq \binom{n-4}{4}$ it is enough to show that
$$\alpha_4-4\alpha_3\leq \binom{n}{4}-4\binom{n}{3}.$$
Moreover, if $\alpha_2\leq \binom{n}{2}-A$, then it is enough to show that
$$\alpha_4-4\alpha_3 - 10A\leq \binom{n}{4}-4\binom{n}{3}.$$
If $\alpha_3=\binom{n}{3}$ then there is nothing to prove, hence we may assume $\alpha_3=\binom{n_3}{3}+\binom{n_2}{2}+\binom{n_1}{1}$
with $n>n_3>n_2>n_1\geq 0$. We consider the functions:
$$f(x)=\binom{x}{4}-4\binom{x}{3},\;g(x)=\binom{x}{3}-4\binom{x}{2}\text{ and }h(x)=\binom{x}{2}-4x.$$
Since $\alpha_4\leq \binom{n_3}{4}+\binom{n_2}{3}+\binom{n_1}{2}$, in order to get the required result, it is enough
to show that 
\begin{equation}\label{pohtit}
f(n_3)+g(n_2)+h(n_1)\leq f(n).
\end{equation}
We have the following table of values:

\begin{center}
\begin{table}[tbh]
\begin{tabular}{|c|c|c|c|c|c|c|c|c|c|c|}
\hline
$x$    & 1 & 2  & 3  & 4  & 5  & 6  & 7  & 8   & 9   & 10  \\ \hline
$f(x)$ & 0 & 0  & -4 & -15& -35& -65&-105& -154& -210& -270 \\ \hline
$g(x)$ & 0 & -4 & -11& -20& -30& -40& -49& -56 & -60 & -60 \\ \hline
$h(x)$ & -4& -7 & -9 & -10& -10& -9 & -7 & -4  &  0  &  5 \\ \hline
\end{tabular}
\end{table}

\begin{table}[tbh]
\begin{tabular}{|c|c|c|c|c|c|c|c|c|c|c|c|}
\hline
$x$    &  11 & 12  & 13  & 14  & 15  & 16  & 17  & 18  & 19  \\ \hline
$f(x)$ & -330&-385 & -429& -455& -455& -420& -340& -204& 0 \\ \hline
$g(x)$ & -55 & -44 & -26 & 0   & 35  & 80  & 136 & 204 & 285 \\ \hline
$h(x)$ &  11 & 18  & 26  & 35  & 45  & 56  & 68  & 81  & 95 \\ \hline
\end{tabular}
\end{table}
\end{center}

As $f(19)=0$, if $n\geq 20$ then 
$$f(n)-(f(n_3)+g(n_2)+h(n_1))\geq f(n)-(f(n-1)+g(n-2)+h(n-3))=n-7\geq 0,$$
and thus we are done. So, we can assume $n\leq 19$.
\begin{itemize}
\item $n=19$. From the table of values of $f,g$ and $h$ it follows that
$$f(n_3)+g(n_2)+h(n_3)\leq f(18)+g(17)+h(16)=-12 < f(19)=0,$$
and thus \eqref{pohtit} is satisfied.
\item $n=18$. 
      From the table of values of $f,g$ and $h$ it follows that
			$$f(n_3)+g(n_2)+h(n_1)\leq f(17)+g(16)+h(15)= -215 < -204 = f(18).$$			
\item $n=17$. 
      From the table of values of $f,g$ and $h$ it follows that
			$$f(n_3)+g(n_2)+h(n_1)\leq f(16)+g(15)+h(14)= -350 < -340 = f(17).$$
\item $n=16$. From \eqref{condi8} it follows that $\alpha_3\geq 224 \geq \binom{12}{3}$.
      If $\alpha_2\leq 116 = \binom{16}{2}-4$ then
			$$\beta_4^7 \leq f(12)-4\cdot 10 = -425 < -420=f(16),$$
			and we are done. If $\alpha_2\geq 117$ then from \eqref{condi8} it follows that
			$$\alpha_3 \geq 6\cdot 117-21\cdot 16+56 = 422 \geq \binom{14}{3}.$$
			It follows that $$\beta_4^7\leq f(14)=-455<-420=f(16),$$
			and thus we are done again.
\item $n=15$. From \eqref{condi8} it follows that $\alpha_3\geq 203 =\binom{11}{3}+\binom{9}{2}+\binom{2}{1}$.
      If $\alpha_2\leq 98 =\binom{15}{2}-7$ then
			$$\beta_4^7 \leq f(11)+g(9)-7\cdot 10 = -460 < -455=f(15),$$
			and we are done. Thus, we may assume $\alpha_2\geq 99$. From \eqref{condi8} it follows that
			$$\alpha_3\geq 6\cdot 99 - 21\cdot 15 + 56 = 335 = \binom{13}{3}+\binom{10}{2}+\binom{4}{1}.$$
			It follows that $\beta_4^7 \leq f(14) = -455 = f(15)$ and we are done.
\item $n=14$. From \eqref{condi8} it follows that $\alpha_3\geq 182=\binom{11}{3}+\binom{6}{2}+\binom{2}{1}$.
      If $\alpha_2\leq 83 =\binom{14}{2}-8$ then
			$$\beta_4^7 \leq f(11)+g(6)+h(2)-8\cdot 10 = -457 < -455=f(14),$$
			and thus we may assume $\alpha_2\geq 84$. From \eqref{condi8} it follows that
			$$\alpha_3\geq 6\cdot 84 - 21\cdot 14 + 56 = 266 = \binom{12}{3}+\binom{10}{2}+\binom{1}{1}.$$
			It follows that $f(n_3)+g(n_2)+h(n_1)\leq f(13)=-429$ and thus, if $\alpha_2\leq 88$ then
			$$\beta_4^7 \leq -459 < -455=f(14),$$
			as required. Now assume $\alpha_2\geq 89$. From \eqref{condi8} it follows that 
			$$\alpha_3\geq 296=\binom{13}{3}+\binom{5}{2}.$$
			Therefore $\beta_4^7\leq f(13)+g(5)=-459 < -455=f(14)$ and we are done.			
\item $n=13$. From \eqref{condi8} it follows that $\alpha_3\geq 161=\binom{10}{3}+\binom{9}{2}+\binom{5}{1}$.

      If $\alpha_2\leq  \binom{13}{2}-10=68$ then 
			$$\beta_4^7 \leq f(11)-10\cdot 10 = -430 < -429=f(13),$$
			and we are done. Assume $\alpha_2\geq 69$. From \eqref{condi8} it follows that
			$$\alpha_3\geq 6\cdot 69 - 21\cdot 13 + 56 = 197 = \binom{11}{3}+\binom{8}{2}+\binom{4}{1}.$$
			If $\alpha_2\leq  \binom{13}{2}-5=73$ then 
			$$\beta_4^7 \leq f(12)-5\cdot 10 = -435 < -429=f(13),$$
			and we are done again. Assume $\alpha_2\geq 74$. From \eqref{condi8} it follows that
			$$\alpha_3\geq 6\cdot 74 - 21\cdot 13 + 56 = 227 = \binom{12}{3}+\binom{4}{2}+\binom{1}{1}.$$
			If $\alpha_2\leq  \binom{13}{2}-2=76$ then
			$$\beta_4^7 \leq f(12)+g(4)+h(1)-2\cdot 10 = -432 < -429=f(13),$$
			so we can assume $\alpha_2\in \{77,78\}$. From \eqref{condi8} it follows that
			$$\alpha_3\geq 6\cdot 77 - 21\cdot 13 + 56 = 245 = \binom{12}{3}+\binom{7}{2}+\binom{4}{1}.$$
			It follows that $\beta_4^7 \leq f(12)+g(11)=-440<f(13)$ and we are done.
\item $n=12$. From \eqref{condi8} it follows that $\alpha_3\geq 140=\binom{10}{3}+\binom{6}{2}+\binom{5}{1}$.

      If $\alpha_2\leq \binom{12}{2}-7=59$ then
			$$\beta_4^7 \leq f(10)+g(7)-7\cdot 10 = -389 < -385 = f(12),$$
			and we are done. Assume $\alpha_2\geq 60$. From \eqref{condi8} it follows that
			$$\alpha_3 \geq 6\cdot 60 -21\cdot 12 + 56 = 164 = \binom{11}{3}-1.$$
			If $\alpha_2=60$ then $\beta_4^7 \leq f(11)-60=-390<f(12)$ and we are done. Hence we may
			assume $\alpha_2\geq 61$ and therefore, from  \eqref{condi8}, we get
			$$\alpha_3\geq 170=\binom{11}{3}+\binom{3}{2}+\binom{2}{1}.$$
			If $\alpha_2\leq 62$ then 
			$$\beta_4^7 \leq f(11)+g(3)+h(2)-4\cdot 10=-388<-385=f(12),$$
			and we are done. Assume $\alpha_2\geq 63$. From  \eqref{condi8} we get 
			$$\alpha_3\geq 182=\binom{11}{3}+\binom{6}{2}+\binom{2}{1}.$$
			If $\alpha_2\leq 65$ then $$\beta_4^7\leq f(11)+g(6)+h(2)-10=-385<f(12),$$ and we are done.
			Finally, if $\alpha_2=66$ then, from  \eqref{condi8}, we have 
			$$\alpha_3\geq 188=\binom{11}{3}+\binom{7}{2}+\binom{2}{1}.$$
			It follows that $$\beta_4^7\leq f(11)+g(7)+h(2)=-386<-385=f(12),$$ and we are done.			
\item $n=11$. From \eqref{condi8} it follows that $\alpha_3\geq 119=\binom{10}{3}-1$. 

      If $\alpha_2\leq \binom{11}{2}-6=49$
      then $$\beta_4^7 \leq f(10) - 60 = -330 = f(11),$$ and we are done, so we may assume $\alpha_2\geq 50$.
			From \eqref{condi8} it follows that $$\alpha_3\geq 125 = \binom{10}{3}+\binom{3}{2}+\binom{2}{1}.$$
			If $\alpha_2=50$ then $$\beta_4^7 \leq f(10)+g(3)+h(2)-50=-338<f(11),$$ and we are done, so we may assume $\alpha_2\geq 51$
			and therefore, from  \eqref{condi8}, we have $$\alpha_3\geq 131 = \binom{10}{3}+\binom{5}{2}+\binom{1}{1}.$$ 
			If $\alpha_2\leq 52$ then
			$$\beta_4^7 \leq f(10)+g(5)+h(1)-30=-334<f(11),$$ and we are done, so we may assume $\alpha_2\geq 53$.
			From  \eqref{condi8} it follows that $$\alpha_3\geq 143 = \binom{10}{3}+\binom{7}{2}+\binom{2}{1}.$$ 
			If $\alpha_2\leq 54$
			then $$\beta_4^7\leq f(10)+g(7)+h(2)-10 = -336<f(11),$$ and we are done. Finally, if $\alpha_2=55$
			then  \eqref{condi8} gives $$\alpha_3\geq 155= \binom{10}{3}+\binom{8}{2}+\binom{7}{1}.$$ It follows that
			$$\beta_4^7\leq f(10)+g(9)=-330=f(11),$$ and we are done.
\item $n=10$. From \eqref{condi8} it follows that $\alpha_3\geq 98=\binom{9}{3}+\binom{5}{2}+\binom{4}{1}$.

      If $\alpha_2\leq \binom{10}{2}-3=42$ then $$\beta_4^7 \leq f(9)+g(5)+h(4)-30=-277<-270=f(10),$$ and we are done,
			so we may assume $\alpha_2\geq 43$. From \eqref{condi8} it follows that 
			$$\alpha_3\geq 104 =\binom{9}{3}+\binom{6}{2}+\binom{5}{1}.$$
			If $\alpha_2=43$ then $$\beta_4^7\leq f(9)+g(7)-20=-279<-270=f(10),$$ and we are done. 
			Assume $\alpha_2\geq 44$.
			Then, by  \eqref{condi8}, we have $$\alpha_3\geq 110=\binom{9}{3}+\binom{7}{2}+\binom{5}{1}.$$ 
			If $\alpha_2=44$ then 
			$$\beta_4^7 \leq f(9)+g(8)-10=-276<f(10),$$ and we are done. 
			Finally, if $\alpha_2=45$ then, by  \eqref{condi8}, we have $$\alpha_3\geq 116=\binom{9}{3}+\binom{8}{2}+\binom{4}{1}.$$
			It follows that $$\beta_4^7 \leq f(9)+g(8)+h(7)=-273<-270=f(10),$$ and we are done again.
\end{itemize}
Thus, the proof is complete.
\end{proof}

\begin{lema}\label{q5_8}
If $q=8$ then $\beta_5^{7} \leq \binom{n-3}{5}$.
\end{lema}

\begin{proof}
Since $\beta_5^7=\beta_5^6-\beta_4^6\leq \beta_5^6$,
$\beta_5^6 = \beta_5^5 - \beta_4^5 \leq \beta_5^5$ and $\beta_5^5 = \alpha_5-\beta_4^4 \leq \alpha_5$,
it follows that we can assume $\alpha_5\geq \binom{n-3}{5}+1$, otherwise there is nothing to prove.
Therefore, we can assume that
\begin{equation}\label{codita}
\alpha_k \geq \binom{n-3}{k}+\binom{4}{k}\text{ for all }2\leq k\leq 5.
\end{equation}
Since $\beta_2^8\geq 0$, $\beta_3^8\geq 0$ and $\beta_4^8\geq 0$ it follows that 
\begin{equation}\label{condy}
\begin{split}
& \alpha_2 \geq 7(n-4),\;\alpha_3 \geq 6\alpha_2-21n+56 \geq 21n-112 \\
& \alpha_4 \geq 5\alpha_3 - 15\alpha_2 + 35n - 70 \geq 35n-210.
\end{split}
\end{equation}
From \eqref{condy} it follows that
\begin{equation}\label{ccondy}
6\alpha_3-10\alpha_2  = 6\alpha_3 - 18\alpha_2 + 8\alpha_2 \leq \frac{6}{5}(\alpha_4+70-35n)+8\binom{n}{2}= \frac{6}{5}\alpha_4 + 84 + 4n^2 - 46 n.
\end{equation}
Moreover, if $\alpha_2=\binom{n}{2}-A$ then, as in \eqref{ccondy} we have
\begin{equation}\label{cccon}
6\alpha_3-10\alpha_2 \leq \frac{6}{5}\alpha_4 + 84 + 4n^2 - 46 n - 8A.
\end{equation}
On the other hand, we have 
\begin{equation}\label{b57}
\beta_5^7 = \alpha_5 - 3\alpha_4 + 6\alpha_3 - 10\alpha_2 + 15n - 21.
\end{equation}
We claim that: 
\begin{equation}\label{klam}
6\alpha_3-10\alpha_2\leq 6\binom{n}{3}-10\binom{n}{2}.
\end{equation}
If $\alpha_2=\binom{n}{2}$ then \eqref{klam} is obvious. Hence, we may assume $\alpha_2=\binom{n_2}{2}+\binom{n_1}{1}$
with $n>n_2>n_1\geq 0$. It follows that $\alpha_3\leq \binom{n_2}{3}+\binom{n_1}{2}$. We consider the functions:
$$f(x)=6\binom{x}{3}-10\binom{x}{2}=x(x-1)(x-7)\text{ and }g(x)=6\binom{x}{2}-10\binom{x}{1}=x(3x-13).$$
Since $n\geq 10$ it follows that
$$6\alpha_3-10\alpha_2\leq f(n_2)+g(n_1)\leq f(n-1)+g(n-2)=f(n)-6n+22\leq f(n),$$
so the claim \eqref{klam} is true. From \eqref{b57} and \eqref{klam} it follow that,
in order to prove that $\beta_5^7\leq \binom{n-3}{5}$, it is enough to show that
$$\alpha_5-3\alpha_4 \leq \binom{n}{5}-3\binom{n}{4}.$$
If $\alpha_4=\binom{n}{4}$ then there is nothing to prove, so we may assume 
$$\alpha_4=\binom{n_4}{4}+\binom{n_3}{3}+\binom{n_2}{2}+\binom{n_1}{1}\text{ where }n>n_4>n_3>n_2>n_1\geq 0,$$
and, also, by \eqref{codita}, that $n_4\geq n-3$. We consider the functions:
$$f_k(x)=\binom{x}{k}-3\binom{x}{k-1},\;\text{ for }2\leq k\leq 5.$$
From the above considerations and Lemma \ref{cord}, in order to complete the proof, 
it is enough to show that 
$$f_5(n_4)+f_4(n_3)+f_3(n_2)+f_2(n_1)\leq f_5(n).$$
We have the following table of values for the functions $f_k(x)$'s:

\begin{center}
\begin{table}[tbh]
\begin{tabular}{|c|c|c|c|c|c|c|c|c|c|c|}
\hline
$x$    & 1 & 2  & 3  & 4  & 5  & 6  & 7  & 8   & 9   & 10  \\ \hline
$f_5(x)$ & 0 & 0  & 0 & -3& -14& -39& -84& -154& -252& -378 \\ \hline
$f_4(x)$ & 0 & 0  & -3& -11& -25& -45& -70& -98& -126& -150\\ \hline
$f_3(x)$ & 0 & -3 & -8& -14& -20& -25& -28& -28& -24 & -15 \\ \hline
$f_2(x)$ & -3& -5 & -6& -6 & -5 & -3 & 0  & 4  & 9   & 15 \\ \hline
\end{tabular}
\end{table}

\begin{table}[tbh]
\begin{tabular}{|c|c|c|c|c|c|c|c|c|c|c|c|}
\hline
$x$      &  11 & 12  & 13  & 14   & 15   & 16   & 17  & 18  & 19  \\ \hline
$f_5(x)$ & -528& -693& -858& -1001& -1092& -1092& -952& -612& 0  \\ \hline
$f_4(x)$ & -165& -165& -143& -91  & 0    & 140  & 340 & 612 & 969 \\ \hline
$f_3(x)$ & 0   & 22  & 52  & 91   & 140  & 200  & 272 & 357 & 456 \\ \hline
$f_2(x)$ & 22  & 30  & 39  & 49   & 60   & 72   & 85  & 99  & 114 \\ \hline
\end{tabular}
\end{table}
\end{center}

If $n\geq 20$ then 
\begin{align*}
& f_5(n)-(f_5(n_4)+f_4(n_3)+f_3(n_2)+f_2(n_1))\geq \\
& f_5(n)-(f_5(n-1)+f_4(n-2)+f_3(n-3)+f_2(n-4)) = n-7 \geq 0,
\end{align*}
and thus we are done. Similarly, if $n=19$ it is easy to see from the table of values that 
$$f_5(n_4)+f_4(n_3)+f_3(n_2)+f_2(n_1)\leq f_5(18)+f_4(17)+f_3(16)+f_2(15)= -12<f_5(19)=0,$$ as required. Hence, we may assume $n\leq 18$.
\begin{itemize}
\item $n=18$. From \eqref{codita} it follows that $\alpha_4\geq \binom{15}{4}+4$.
      From the table of values of $f_k(x)$'s it follows that
      $$\alpha_5-3\alpha_4\leq f_5(17)+f_4(16)+f_3(15)+f_2(14) = -623 < f_5(18)=-612,$$
      as required.
\item $n=17$. From \eqref{codita} it follows that $\alpha_4\geq \binom{14}{4}+4$.
      From the table of values of $f_k(x)$'s it follows that
      $$\alpha_5-3\alpha_4\leq f_5(16)+f_4(15)+f_3(14)+f_2(13)=-962 < f_5(17)=-952,$$
\item $n=16$. From \eqref{codita} it follows that $\alpha_4\geq \binom{13}{4}+4$.
      If $\alpha_4\geq \binom{14}{4}+\binom{7}{3}+\binom{5}{2}+\binom{1}{1}$, then, from the
			table of values of $f_k(x)$'s it follows that 
			$$\alpha_5-3\alpha_4 \leq f_5(14)+f_4(7)+f_3(5)+f_2(1) = -1094  < -1092 = f_5(16).$$
			Hence, we may assume $\alpha_4\leq \binom{14}{4}+\binom{7}{3}+\binom{5}{2}=1046$.
			From \eqref{ccondy} it follows that
			$$6\alpha_3-10\alpha_2 \leq \frac{6}{5}\cdot 1046 + 84 - 42\cdot 16 +64\cdot 15 = 1627.2$$
			Hence, in order to prove that $\beta_5^7\leq \binom{13}{5}$ it is enough to show
			that $$\alpha_5-3\alpha_4 \leq \binom{13}{5}-1628-15\cdot 16+21 = -560,$$ 
			which is true, as $\alpha_4\geq \binom{13}{4}$ and $f_5(13)=-858<-560$.
\item $n=15$. From \eqref{codita} it follows that $\alpha_4\geq \binom{12}{4}+4$. As in the case $n=16$,
      we can assume $\alpha_4\leq 1046$. From \eqref{ccondy} it follows that
			$$6\alpha_3-10\alpha_2 \leq \frac{6}{5}\cdot 1046 + 84 - 42\cdot 15 + 60\cdot 14 = 1549.2$$
      Hence, in order to prove that $\beta_5^7\leq \binom{12}{5}$ it is enough to show
			that $$\alpha_5-3\alpha_4 \leq \binom{12}{5}-1550-15\cdot 15+21 = -962.$$
			Since $f_5(13)+f_4(8)+f_3(2)+f_2(1)=-858-98-3-3=-962$,
			this condition is satisfied if $\alpha_4\geq \binom{13}{4}+\binom{8}{3}+\binom{2}{2}+\binom{1}{1}$
			and thus we may assume $$\alpha_4\leq \binom{13}{4}+\binom{8}{3}+\binom{2}{2}=772.$$ From \eqref{ccondy} it follows that
			$$6\alpha_3-10\alpha_2 \leq \frac{6}{5}\cdot 772 + 84 - 42\cdot 15 + 60\cdot 14 =  1220.2$$
			Hence, we have to prove that $\alpha_5-3\alpha_4 \leq -633$, which is true as 
			$f_5(12)=-693$ and $\alpha_4\geq \binom{12}{4}$.
\item $n=14$. From \eqref{codita} it follows that $\alpha_4\geq \binom{11}{4}+4$. In the following, we will skip some obvious steps
      in the argument, as they are similar to the previous cases.
			
      If $\alpha_4\geq \binom{13}{4}+\binom{9}{3}+\binom{4}{2}+\binom{1}{1}$ then $\alpha_5-3\alpha_4\leq f_5(14)$,
			hence we may assume $\alpha_4 \leq \binom{13}{4}+\binom{9}{3}+\binom{4}{2}=805$.
      From \eqref{ccondy} it follows that $6\alpha_3-10\alpha_2 \leq 1190$. Hence, in order to prove that $\beta_5^7\leq \binom{11}{5}$
			it is enough to show that $$\alpha_5-3\alpha_4 \leq \binom{11}{5}-1190-15\cdot 14 + 21 = -917.$$
			This condition holds for $\alpha_4\geq \binom{13}{4} + \binom{6}{3}+\binom{4}{2}=741$, hence we may assume
			$\alpha_4\leq 740$. From \eqref{ccondy} it follows that $6\alpha_3-10\alpha_2 \leq 1112$ and thus it is enough
			to show that $\alpha_5-3\alpha_4\leq - 839$. This is true for 
			$\alpha_4\geq \binom{12}{4}+\binom{9}{3}+\binom{4}{2}+\binom{3}{1}=588$ so we may assume $\alpha_4\leq 587$.
			From \eqref{ccondy} it follows that it is enough to show $\alpha_5-3\alpha_4\leq -656$. This is true for 
			$\alpha_4\geq \binom{11}{4}+\binom{9}{3}+\binom{2}{2}=415$ and thus we may assume $\alpha_4\leq 414$.
			From \eqref{ccondy} it follows that it is enough to show $\alpha_5-3\alpha_4\leq -448$, which is true since 
			$\alpha_4\geq \binom{11}{4}$.
			
\item $n=13$. From \eqref{condy} it follows that $\alpha_4\geq 245=\binom{10}{4}+\binom{7}{3}$.
      If $\alpha_4\geq \binom{12}{4}+\binom{10}{3}+\binom{4}{2}+\binom{1}{1}$ then $\alpha_5-3\alpha_4\leq f_5(13)$,
			hence we may assume $\alpha_4 \leq \binom{12}{4}+\binom{10}{3}+\binom{4}{2}=621$. From \eqref{ccondy} it follows that
			$$6\alpha_3-10\alpha_2 \leq \frac{6}{5}\cdot 621 + 84 + 4\cdot 13^2 - 46\cdot 13 = 907.2$$
			Hence, in order to prove that $\beta_5^7\leq \binom{10}{5}$ it is enough to show that
			$$\alpha_5-3\alpha_4 \leq \binom{10}{5}-908-15\cdot 13 + 21 = -830.$$
			This holds for $\alpha_4\geq \binom{12}{4}+\binom{9}{3}+\binom{3}{2}=582$ so we may assume $\alpha_4\leq 581$.
			
			From \eqref{ccondy} it follows that it is enough to show that $\alpha_5-3\alpha_4\leq -782$.
			This holds for $\alpha_4\geq \binom{12}{4}+\binom{7}{3}+\binom{4}{2}+\binom{2}{1}=538$ so we may assume $\alpha_4\leq 537$.
			
			From \eqref{ccondy} it follows that it is enough to show that $\alpha_5-3\alpha_4\leq -729$.
			This holds for $\alpha_4\geq \binom{12}{4}+\binom{5}{3}+\binom{3}{2}+\binom{1}{1}=509$ so we may assume $\alpha_4\leq 508$.
			
			From \eqref{ccondy} it follows that it is enough to show that $\alpha_5-3\alpha_4\leq -694$. This is true for 
			$\alpha_4>\binom{12}{4}$ so we may assume $\alpha_4\leq \binom{12}{4}=495$.
			
			From \eqref{ccondy} it follows that it is enough to show that $\alpha_5-3\alpha_4\leq -678$.
			This holds for $\alpha_4 \geq \binom{11}{4}+\binom{9}{3}+\binom{5}{2}+\binom{2}{1}=426$ so we may assume $\alpha_4\leq 425$.
			
			From \eqref{ccondy} it follows that it is enough to show that $\alpha_5-3\alpha_4\leq -594$.
			This holds for $\alpha_4 \geq \binom{11}{4}+\binom{6}{3}+\binom{5}{2}+\binom{1}{1}=361$ so we may assume $\alpha_4\leq 360$.
			
			From \eqref{ccondy} it follows that it is enough to show that $\alpha_5-3\alpha_4\leq -516$.
			This holds for $\alpha_4 \geq \binom{10}{4}+\binom{9}{3}+\binom{3}{2}+\binom{2}{1}=299$ so we may assume $\alpha_4\leq 298$.
			
			From \eqref{ccondy} it follows that it is enough to show that $\alpha_5-3\alpha_4\leq -442$. And this is true, as
			$\alpha_4\geq \binom{10}{4}+\binom{7}{3}$.
			
\item $n=12$. From \eqref{condy} it follows that $\alpha_4\geq 210=\binom{10}{4}$.
      If $\alpha_4\geq \binom{11}{4}+\binom{10}{3}+\binom{4}{2}+\binom{1}{1}$ then $\alpha_5-3\alpha_4\leq f_5(12)$,
			hence we may assume $\alpha_4 \leq \binom{11}{4}+\binom{10}{3}+\binom{4}{2}=456$. From \eqref{ccondy} it follows that
			$$6\alpha_3-10\alpha_2 \leq \frac{6}{5}\cdot 456 + 84 + 4\cdot 12^2 - 46\cdot 12 = 655.2$$
			Hence, in order to prove that $\beta_5^7\leq \binom{9}{5}$ it is enough to show that
			$$\alpha_5-3\alpha_4 \leq \binom{9}{5}-656-15\cdot 12 + 21 = -689.$$
			This condition holds for $\alpha_4\geq \binom{11}{4}+\binom{10}{3}+\binom{3}{2}+\binom{1}{1}$ so
			we may assume $\alpha_4\leq 453$.
			
			From \eqref{ccondy} it follows that it is enough to show that $\alpha_5-3\alpha_4\leq -685$.
			This holds for $\alpha_4 = 453$ so we may assume $\alpha_4\leq 452$. This means that we need to show that
			$\alpha_5-3\alpha_4\leq -684$ which is true for $\alpha_4=452$. So we may assume $\alpha_4\leq 451$ and
			therefore we need to show that $\alpha_5-3\alpha_4\leq -683$.
			
      Assume $\alpha_4=451$. If $\alpha_2 < \binom{12}{2}$, that is $\alpha_2\leq \binom{11}{2}+\binom{10}{1}$, it
			follows that $\alpha_4\leq \binom{11}{4}+\binom{10}{3}=450$, a contradiction. Hence  $\alpha_2=\binom{12}{2}$.
			Since $\beta_4^8\geq 0$ and $\alpha_4\leq 451$ it follows that $\alpha_3\leq \binom{12}{3}-2$,
			otherwise $\beta_4^8<0$, a contradiction.	By straightforward computations, it follows that 
			$$\beta_5^7 = \alpha_5-3\alpha_4 + 6\binom{12}{3}-12 - 10\binom{12}{2} + 15\cdot 12 - 21 \leq 126=\binom{9}{5},$$
			as required. So, we may assume $\alpha_4\leq 450$ and we need to show that $\alpha_5-3\alpha_4\leq -681$.
			
			Assume $\alpha_4=450=\binom{11}{4}+\binom{10}{3}$. It follows that $\alpha_2\geq \binom{12}{2}-1$. Assume $\alpha_2=\binom{12}{2}$.
			As in the case $\alpha_4=451$, from $\beta_4^8\geq 0$ we deduce $\alpha_3\leq \binom{12}{3}-2$. If $\alpha_3\leq \binom{12}{3}-3$
			then we can show that $\beta_5^7\leq 123$. On the other hand, if $\alpha_3 = \binom{12}{3}-2$ then $\alpha_5\leq \binom{11}{5}+\binom{10}{4}-6$,
			otherwise $\beta_6^8<0$, a contradiction. Therefore $\beta_5^7\leq 123$. On the other hand, if $\alpha_2=\binom{12}{2}-1=\binom{11}{2}+\binom{10}{1}$
			if follows that $\alpha_3 = \binom{11}{3}+\binom{10}{2}$. Therefore $\beta_5^7\leq 91 < 126$.
			
			So, we can assume $\alpha_4\leq 449$ and we have to show $\alpha_5-3\alpha_4\leq -680$. Assume $\alpha_4\geq 448$, that is 
			$\alpha_4\in \{448,449\}$. Note that $\alpha_2\geq \binom{12}{2}-2$. If $\alpha_2=\binom{12}{2}$ then, from $\beta_4^8\geq 0$ 
			it follows that $\alpha_3\leq \binom{12}{3}-3$. Since $\alpha_5\leq \binom{11}{5}+\binom{9}{4}+\binom{8}{3}+\binom{7}{2}$
			it follows that $\beta_5^7\leq 122$. On the other hand, if $\alpha_2\leq \binom{12}{2}-1$ then $\alpha_3\leq \binom{11}{3}+\binom{10}{2}$
			and we have $\beta_5^7\leq 90$.
			
			So, we may assume $\alpha_4\leq 447$ and we need to show that $\alpha_5-3\alpha_4\leq -678$. This is true for
			$\alpha_4\geq \binom{11}{4}+\binom{9}{3}+\binom{5}{2}+\binom{2}{1}=426$ so we may assume $\alpha_4\leq 425$.
			
			From \eqref{ccondy} it follows that we have to show that $\alpha_5-3\alpha_4\leq -651$. This is true for
			$\alpha_4\geq  \binom{11}{4}+\binom{8}{3}+\binom{5}{2}+\binom{2}{1}=398$ so we may assume $\alpha_4\leq 397$.
			
			From \eqref{ccondy} it follows that we have to show that $\alpha_5-3\alpha_4\leq -618$. This is true for
			$\alpha_4\geq  \binom{11}{4}+\binom{7}{3}+\binom{4}{2}+\binom{3}{1}=374$ so we may assume $\alpha_4\leq 373$.
			
			From \eqref{ccondy} it follows that we have to show that $\alpha_5-3\alpha_4\leq -589$. This is true for
			$\alpha_4\geq  \binom{11}{4}+\binom{6}{3}+\binom{4}{2}+\binom{1}{1}=357$ so we may assume $\alpha_4\leq 356$.

      From \eqref{ccondy} it follows that we have to show that $\alpha_5-3\alpha_4\leq -569$. This is true for
			$\alpha_4\geq  \binom{11}{4}+\binom{5}{3}+\binom{4}{2}+\binom{1}{1}=347$ so we may assume $\alpha_4\leq 346$.
			
			From \eqref{ccondy} it follows that we have to show that $\alpha_5-3\alpha_4\leq -557$. This is true for
			$\alpha_4\geq  \binom{11}{4}+\binom{5}{3}+\binom{2}{2}+\binom{1}{1}=342$ so we may assume $\alpha_4\leq 341$.
			
			From \eqref{ccondy} it follows that we have to show that $\alpha_5-3\alpha_4\leq -551$. This is true for
			$\alpha_4\geq  \binom{11}{4}+\binom{4}{3}+\binom{3}{2}+\binom{2}{1}=339$ so we may assume $\alpha_4\leq 338$.
			
			From \eqref{ccondy} it follows that we have to show that $\alpha_5-3\alpha_4\leq -547$. This is true for
			$\alpha_4\geq  \binom{11}{4}+\binom{4}{3}+\binom{3}{2}=337$ so we may assume $\alpha_4\leq 336$.
			
			From \eqref{ccondy} it follows that we have to show that $\alpha_5-3\alpha_4\leq -545$. This condition holds for $\alpha_4=336$
			so we may assume $\alpha_4\leq 335$ and we have to show that $\alpha_5-3\alpha_4\leq -543$.
			If $\alpha_2=\binom{12}{2}$ then, as $\beta_3^8\geq 0$, it follows that $\alpha_3\geq 200$. Hence
			$\beta_4^8\leq -25$, a contradiction. Similarly, if $\alpha_2=\binom{12}{2}-1$ then $\beta_3^8\geq 0$ forces $\alpha_3\geq 194$
			and we get $\beta_4^8\leq -10$. It follows that $\alpha_2\leq \binom{12}{2}-2$. Since $\alpha_4\leq 335$, from \eqref{cccon}
			it follows that it is enough to show that  $\alpha_5-3\alpha_4\leq -545+16=-529$. This is true for $\alpha_5\geq \binom{11}{4}+1=331$
		  so we may assume $\alpha_5\leq 330$.
			
			From \eqref{cccon} it follows that we have to show that $\alpha_5-3\alpha_4\leq -521$. This is true for
			$\alpha_4\geq \binom{10}{4}+\binom{9}{3}+\binom{4}{2}+\binom{1}{1}=301$, so we may assume $\alpha_5\leq 300$.
			
			From \eqref{cccon} it follows that we have to show that $\alpha_5-3\alpha_4\leq -485$. This is true for
			$\alpha_4\geq \binom{10}{4}+\binom{8}{3}+\binom{3}{2}+\binom{1}{1}=270$, so we may assume $\alpha_5\leq 269$.
			
			From \eqref{cccon} it follows that we have to show that $\alpha_5-3\alpha_4\leq -448$. This is true for
			$\alpha_4\geq \binom{10}{4}+\binom{6}{3}+\binom{5}{2}+\binom{2}{1}=242$, so we may assume $\alpha_5\leq 241$.
			
			From \eqref{cccon} it follows that we have to show that $\alpha_5-3\alpha_4\leq -415$. This is true for
			$\alpha_4\geq \binom{10}{4}+\binom{5}{3}+\binom{3}{2}+\binom{2}{1}= 225$, so we may assume $\alpha_5\leq 224$.
			
			From \eqref{cccon} it follows that we have to show that $\alpha_5-3\alpha_4\leq -394$. This is true for
			$\alpha_4\geq \binom{10}{4}+\binom{4}{3}+\binom{2}{2}+\binom{1}{1}=216$, so we may assume $\alpha_5\leq 215$.
			
			From \eqref{cccon} it follows that we have to show that $\alpha_5-3\alpha_4\leq -383$. This is true for
			$\alpha_4\geq \binom{10}{4}+2=212$, so we may assume $\alpha_5\leq 211$. 
			
			From \eqref{cccon} it follows that we have to show that $\alpha_5-3\alpha_4\leq -379$, a condition which is
			satisfied for $\alpha_4=211$ so we may assume $\alpha_4=210$ (since $\alpha_4\geq 210$) and we have to
			show that $\alpha_5-3\alpha_4\leq -377$ which is satisfied by $\alpha_4=210$.

\item $n=11$. From \eqref{condy} it follows that $\alpha_4\geq 175=\binom{9}{4}+\binom{7}{3}+\binom{5}{2}+\binom{4}{1}$.
      If $\alpha_4\geq \binom{10}{4}+\binom{9}{3}+\binom{5}{2}+\binom{2}{1}=306$ then $\alpha_5-3\alpha_4\leq f_5(11)$,
			hence we may assume $\alpha_4 \leq 305$. From \eqref{ccondy} it follows that
			$$6\alpha_3-10\alpha_2 \leq \frac{6}{5}\cdot 305 + 84 + 4\cdot 11^2 - 46\cdot 11 = 428.$$
			Hence, in order to prove that $\beta_5^7\leq \binom{8}{5}$ it is enough to show that
			$$\alpha_5-3\alpha_4 \leq \binom{8}{5}-428-15\cdot 11 + 21 = -516.$$
			This is true for $\alpha_4\geq \binom{10}{4}+\binom{9}{3}+\binom{3}{2}+\binom{2}{1}=299$ so we may assume $\alpha_4\leq 298$.
			From \eqref{ccondy} it follows that it is enough to show that $\alpha_5-3\alpha_4\leq -508$.
			This is true for $\alpha_4\geq \binom{10}{4}+\binom{9}{3}+\binom{2}{2}+\binom{1}{1}=296$ so we may assume $\alpha_4\leq 295$.
			
			From \eqref{ccondy} it follows that it is enough to show that $\alpha_5-3\alpha_4\leq -504$.
			This is true for $\alpha_4\geq \binom{10}{4}+\binom{8}{3}+\binom{6}{2}+\binom{1}{1}=282$ so we may assume $\alpha_4\leq 281$.
			
			From \eqref{ccondy} it follows that it is enough to show that $\alpha_5-3\alpha_4\leq -488$.
			This is true for $\alpha_4\geq \binom{10}{4}+\binom{8}{3}+\binom{3}{2}+\binom{2}{1}=271$ so we may assume $\alpha_4\leq 270$.
			
			From \eqref{ccondy} it follows that it is enough to show that $\alpha_5-3\alpha_4\leq -474$.
			This is true for $\alpha_4\geq \binom{10}{4}+\binom{7}{3}+\binom{5}{2}+\binom{4}{1}=259$ so we may assume $\alpha_4\leq 258$.
			
			From \eqref{ccondy} it follows that it is enough to show that $\alpha_5-3\alpha_4\leq -460$.
			This is true for $\alpha_4\geq \binom{10}{4}+\binom{7}{3}+\binom{3}{2}+\binom{2}{1}=250$ so we may assume $\alpha_4\leq 249$.
			
			From \eqref{ccondy} it follows that it is enough to show that $\alpha_5-3\alpha_4\leq -449$.
			This is true for $\alpha_4\geq \binom{10}{4}+\binom{7}{3}+\binom{2}{2}=246$ so we may assume $\alpha_4\leq 245$.
			
			From \eqref{ccondy} it follows that it is enough to show that $\alpha_5-3\alpha_4\leq -444$.
			This is true for $\alpha_4\geq \binom{10}{4}+\binom{6}{3}+\binom{5}{2}+\binom{1}{1}=241$ so we may assume $\alpha_4\leq 240$.
			
			From \eqref{ccondy} it follows that it is enough to show that $\alpha_5-3\alpha_4\leq -438$.
      This is true for $\alpha_4\geq \binom{10}{4}+\binom{6}{3}+\binom{4}{2}+\binom{1}{1}=237$ so we may assume $\alpha_4\leq 236$.
			
			From \eqref{ccondy} it follows that it is enough to show that $\alpha_5-3\alpha_4\leq -434$.
			This is true for $\alpha_4\geq \binom{10}{4}+\binom{6}{3}+\binom{3}{2}+\binom{1}{1}=234$ so we may assume $\alpha_4\leq 233$.
			
			From \eqref{ccondy} it follows that it is enough to show that $\alpha_5-3\alpha_4\leq -430$. This is true for $\alpha_4=233$
			so we may assume $\alpha_4\leq 232$. Hence we have to show that $\alpha_5-3\alpha_4\leq -429$. This is true for $\alpha_4=232$
			so we may assume $\alpha_4\leq 231$ and we have to show that $\alpha_5-3\alpha_4\leq -428$.
			
			If $\alpha_2=\binom{11}{2}$ then $\beta_3^8\geq 0$ implies $\alpha_3\geq \binom{11}{3}-10$. Since $\alpha_4\leq 231$ it
			follows that $\beta_4^8\leq -34$, a contradiction. If $\alpha_2=\binom{11}{2}-1$ then $\beta_3^8\geq 0$ implies $\alpha_3\geq \binom{11}{3}-16$.
			And this implies $\beta_4^8\leq -19$, again a contradiction. Similarly, if $\alpha_2=\binom{11}{2}-2$ then $\beta_3^8\geq 0$ implies 
			$\alpha_3\geq \binom{11}{3}-22$ and this lead to $\beta_4^8\leq -4$, again a contradiction. Therefore, we may assume $\alpha_2\leq \binom{11}{2}-3$.
			
			From \eqref{cccon} it follows that we have to show that $\alpha_5-3\alpha_4\leq -404$. This is true for
			$\alpha_4\geq \binom{10}{4}+\binom{5}{3}+\binom{2}{2}=221$ so we may assume $\alpha_4\leq 220$.
			
			From \eqref{cccon} it follows that we have to show that $\alpha_5-3\alpha_4\leq -390$. This is true for
			$\alpha_4\geq \binom{10}{4}+\binom{4}{3}+\binom{2}{2}=215$ so we may assume $\alpha_4\leq 214$.
			
			From \eqref{cccon} it follows that we have to show that $\alpha_5-3\alpha_4\leq -383$. 
			This is true for $\alpha_4\geq \binom{10}{4}+\binom{3}{3}+\binom{2}{2}=212$ so we may assume $\alpha_4\leq 211$.
			
			From \eqref{cccon} it follows that we have to show that $\alpha_5-3\alpha_4\leq -380$. This is true for $\alpha_4\geq 211$
			so we may assume $\alpha_4\leq 210$.
			
			From \eqref{cccon} it follows that we have to show that $\alpha_5-3\alpha_4\leq -378$. This is true for
			$\alpha_4\geq \binom{9}{4}+\binom{8}{3}+\binom{6}{2}+\binom{1}{1}=198$ so we may assume $\alpha_4\leq 197$.
			
			From \eqref{cccon} it follows that we have to show that $\alpha_5-3\alpha_4\leq -363$. This is true for
			$\alpha_4\geq \binom{9}{4}+\binom{8}{3}+\binom{3}{2}+\binom{2}{1}=187$ so we may assume $\alpha_4\leq 186$.
			
			From \eqref{cccon} it follows that we have to show that $\alpha_5-3\alpha_4\leq -350$. This is true for
			$\alpha_4\geq \binom{9}{4}+\binom{7}{3}+\binom{6}{2}+\binom{1}{1}=177$ so we may assume $\alpha_4\leq 176$.
			
			From \eqref{cccon} it follows that we have to show that $\alpha_5-3\alpha_4\leq -338$. This is true, since
			$\alpha_4\geq \binom{9}{4}+\binom{7}{3}+\binom{5}{2}$.
			
\item $n=10$. From \eqref{condy} it follows that $\alpha_4\geq 140=\binom{9}{4}+\binom{5}{3}+\binom{3}{2}+\binom{1}{1}$.
      If $\alpha_4\geq \binom{9}{4}+\binom{8}{3}+\binom{6}{2}+\binom{1}{1}$ then $\alpha_5-3\alpha_4\leq f_5(10)$,
			hence we may assume $\alpha_4 \leq \binom{9}{4}+\binom{8}{3}+\binom{6}{2}=197$. From \eqref{ccondy} it follows that
			$$6\alpha_3-10\alpha_2 \leq \frac{6}{5}\cdot 197 + 84 + 4\cdot 10^2 - 46\cdot 10 = 260.4$$
			Hence, in order to prove that $\beta_5^7\leq \binom{7}{5}$ it is enough to show that
			$$\alpha_5-3\alpha_4 \leq \binom{7}{5}-261-15\cdot 10 + 21 = -369.$$
      This condition holds for $\alpha_4\geq \binom{9}{4}+\binom{8}{3}+\binom{4}{2}+\binom{2}{1}=190$ so we may assume $\alpha_4\leq 189$.
			
			From \eqref{ccondy} it follows that we have to show that $\alpha_5-3\alpha_4\leq -360$. This is true for
			$\alpha_4\geq \binom{9}{4}+\binom{8}{3}+\binom{3}{2}+\binom{1}{1}=186$ so we may assume $\alpha_4\leq 185$.
			
			From \eqref{ccondy} it follows that we have to show that $\alpha_5-3\alpha_4\leq -354$. This is true for
			$\alpha_4\geq \binom{9}{4}+\binom{8}{3}+\binom{2}{2}+\binom{1}{1}=184$ so we may assume $\alpha_4\leq 183$.
			
			From \eqref{ccondy} it follows that we have to show that $\alpha_5-3\alpha_4\leq -352$. This is true for $\alpha_4=183$
			so we may assume $\alpha_4\leq 182$ and we have to show that $\alpha_5-3\alpha_4\leq -351$. 
			
			If $\alpha_2=\binom{10}{2}$ then $\beta_3^8\geq 0$ implies $\alpha_3\geq \binom{10}{3}-4$. Therefore $\beta_4^8\leq -3$, a contradiction.
			Hence $\alpha_2\leq \binom{10}{2}-1$. From \eqref{cccon} it follows that we have to show $\alpha_5-3\alpha_4\leq -343$. This is true for
			$\alpha_4\geq \binom{9}{4}+\binom{7}{3}+\binom{5}{2}+\binom{1}{1}=172$ so we may assume $\alpha_4\leq 171$.
			
			From \eqref{cccon} it follows that we have to show that $\alpha_5-3\alpha_4\leq -330$. This is true for
			$\alpha_4\geq \binom{9}{4}+\binom{7}{3}+\binom{3}{2}=164$ so we may assume $\alpha_4\leq 163$.
			
			From \eqref{cccon} it follows that we have to show that $\alpha_5-3\alpha_4\leq -320$. This is true for
			$\alpha_4\geq \binom{9}{4}+\binom{6}{3}+\binom{5}{2}+\binom{1}{1}=157$ so we may assume $\alpha_4\leq 156$.
			
			From \eqref{cccon} it follows that we have to show that $\alpha_5-3\alpha_4\leq -312$. This is true for
      $\alpha_4\geq \binom{9}{4}+\binom{6}{3}+\binom{4}{2}+\binom{1}{1}=153$ so we may assume $\alpha_4\leq 152$.

      From \eqref{cccon} it follows that we have to show that $\alpha_5-3\alpha_4\leq -307$. This is true for
      $\alpha_4\geq \binom{9}{4}+\binom{6}{3}+\binom{3}{2}+\binom{1}{1}=150$ so we may assume $\alpha_4\leq 149$.
			
			From \eqref{cccon} it follows that we have to show that $\alpha_5-3\alpha_4\leq -303$. This is true for
      $\alpha_4\geq \binom{9}{4}+\binom{6}{3}+\binom{2}{2}+\binom{1}{1}=148$ so we may assume $\alpha_4\leq 147$.
			
			From \eqref{cccon} it follows that we have to show that $\alpha_5-3\alpha_4\leq -301$. Now, if $\alpha_2=\binom{10}{2}-1$ then
			from the fact that $\beta_3^8\geq 0$ it follows that $\alpha_3\geq \binom{10}{3}-10$ and moreover $\beta_4^8\leq -23$, a contradiction.
			Similarly, if $\alpha_2=\binom{10}{2}-2$ we get $\alpha_3\geq \binom{10}{3}-16$ and $\beta_4^8\leq -8$, again a contradiction. 
			Hence, $\alpha_2\leq \binom{10}{3}-3$ and from \eqref{cccon} it is enough to show that $\alpha_5-3\alpha_4\leq -301+16=-285$.
			This condition is satisfied, as $\alpha_4\geq \binom{9}{4}+\binom{5}{3}+\binom{3}{2}$.

\end{itemize}
Thus, the proof is complete.
\end{proof}

\begin{lema}\label{q6_8}
If $q=8$ then $\beta_6^{7} \leq \binom{n-2}{6}$.
\end{lema}

\begin{proof}
Since $\beta_2^8\geq 0$, $\beta_3^8\geq 0$, $\beta_5^8\geq 0$ and $\beta_5^8\geq 0$ it follows that 
\begin{equation}\label{condyy}
\begin{split}
& \alpha_2 \geq 7(n-4),\;\alpha_3 \geq 6\alpha_2-21n+56 \geq 21n-112 \\
& \alpha_4 \geq 5\alpha_3 - 15\alpha_2 + 35n - 70 \geq 35n- 210 \\
& \alpha_5 \geq 4\alpha_4 - 10\alpha_3 + 20\alpha_2 - 35n + 56 \geq  35n - 224
\end{split}
\end{equation}
From \eqref{condyy} it follows
							that \begin{equation}\label{sisi}
							3\alpha_4-4\alpha_3+5\alpha_2 \leq \alpha_5+\alpha_3+14.
							\end{equation}
Since $\beta_6^7 = \beta_6^6 - \beta_5^6 \leq \beta_6^6 \leq \alpha_6$, we can safely assume that $\alpha_6\geq \binom{n-2}{6}+1$, otherwise
there is nothing to prove. In particular, it follows that $\alpha_k\geq \binom{n-2}{k}+\binom{5}{k}$ for all $2\leq k\leq 5$.
We claim that 
\begin{equation}\label{kukuku}
3\alpha_4-4\alpha_3 \leq 3\binom{n}{4}-4\binom{n}{3}.
\end{equation}
If $\alpha_3=\binom{n}{3}$ then there is nothing to prove. Assume $\alpha_3=\binom{n_3}{3}+\binom{n_2}{2}+\binom{n_1}{1}$,
where $n_3>n_2>n_1\geq 0$ and $n_3\geq n-2$. Since $n\geq 10$, it is easy to check that
\begin{align*}
& 3\alpha_4-4\alpha_3\leq 3\binom{n_3}{4}+3\binom{n_2}{3}+3\binom{n_1}{2}-4\binom{n_3}{3}-4\binom{n_2}{2}-4\binom{n_1}{1} \leq \\
& \leq 3\binom{n-1}{4}+3\binom{n-2}{3}+3\binom{n-3}{2}-4\binom{n-1}{3}-4\binom{n-2}{2}-4\binom{n-3}{1} = \\
& = 3\binom{n}{4}-4\binom{n}{3} + 13 - 3n < 3\binom{n}{4}-4\binom{n}{3},
\end{align*}
as required. From \eqref{kukuku} it follows that, in order to prove that $\beta_6^{7} \leq \binom{n-2}{6}$, it
suffice to show that 
$$\alpha_6-2\alpha_5\leq \binom{n}{6}-2\binom{n}{5}.$$
If $\alpha_5=\binom{n}{5}$ then there is nothing to prove, hence we may assume
$$\alpha_5=\binom{n_5}{5}+\binom{n_4}{4}+\binom{n_3}{3}+\binom{n_2}{2}+\binom{n_1}{1},$$
where $n_5>n_4>n_3>n_2>n_1\geq 0$ and $n_5\geq n-2$. Also, if $n_5=n-2$ then $n_4\geq 5$.

We consider the functions $f_k(x)=\binom{x}{k}-2\binom{x}{k-1}$, for $2\leq k\leq 6$. We have:

 \begin{center}
 \begin{table}[tbh]
 \begin{tabular}{|c|c|c|c|c|c|c|c|c|c|c|c|c|c|c|c|c|}
 \hline
 $x$      & 1  & 2  & 3  & 4  & 5  & 6  & 7   & 8    & 9   & 10  & 11  & 12  & 13  & 14   & 15   & 16    \\ \hline
 $f_6(x)$ & 0  & 0  & 0  & 0  & -2 & -11& -35 & -84  & -168& -294& -462& -660& -858& -1001& -1001& -728  \\ \hline
 $f_5(x)$ & 0  & 0  & 0  & -2 & -9 & -24& -49 & -84  & -126& -168& -198& -198& -143& 0    & 273  &    \\ \hline
 $f_4(x)$ & 0  & 0  & -2 & -7 & -15& -25& -35 & -42  & -42 & -30 & 0   & 55  & 143 & 273  &   &       \\ \hline
 $f_3(x)$ & 0  & -2 & -5 & -8 & -10& -10& -7  & 0    & 12  & 30  & 55  & 88  & 130 &   &      &       \\ \hline
 $f_2(x)$ & -2 & -3 & -3 & -2 & 0  & 3  & 7   & 12   &18   & 25  & 33  & 42  &   & & &  \\ \hline
 \end{tabular}
 \end{table}
\end{center}

If $n\geq 16$ then, from the above table and the conditions on $\alpha_5$, we have
$$f_6(n)-(\alpha_6-2\alpha_5) \geq f_6(n)-(f_6(n-1)+f_5(n-2)+f_4(n-3)+f_3(n-4)+f_2(n-5))=n-7,$$ 
and we are done. We have to consider the following cases:
\begin{itemize}
\item $n=15$. If $\alpha_5\geq \binom{13}{5}+\binom{9}{4}+\binom{5}{3}+\binom{2}{2}=1424$ then
              $$\alpha_6-2\alpha_5\leq f_6(13)+f_5(9)+f_4(5)+f_3(2) = -1001 =f_6(15),$$
							and we are done. Hence, we may assume $\alpha_5\leq 1423$. 
							Since $\alpha_5\leq 1423$ and $\alpha_3\leq \binom{15}{3}$, from \eqref{sisi} it follows that 
							$$3\alpha_4-4\alpha_3+5\alpha_2 \leq 1892.$$
							Thus, in order to show that $$\beta_6^7 = \alpha_6-2\alpha_5+ (3\alpha_4-4\alpha_3+5\alpha_2) - 6n + 7 \leq \binom{13}{6}=1716,$$ it suffice to show
							$\alpha_6-2\alpha_5 \leq -44$, which is true, as $\alpha_5\geq \binom{13}{5}$ and $f_6(13)=-858$.
\item $n=14$. As in the case $n=15$, if $\alpha_5\geq 1424$ then we are done, so we may assume $\alpha_5\leq 1423$.
              Since $\alpha_3\leq \binom{14}{3}$, from \eqref{sisi} it follows that
							$$3\alpha_4-4\alpha_3+5\alpha_2 \leq 1801.$$
							Thus, in order to show that $\beta_6^7\leq \binom{12}{6}=924$, it is enough to show
							that $\alpha_6-2\alpha_5 \leq -800$, which is true, for $\alpha_5\geq \binom{12}{5}+\binom{10}{4}+\binom{5}{3}=1012$.
							Now, assume $\alpha_5\leq 1011$. Using \eqref{sisi} again, it follows that it is enough to show $\alpha_6-2\alpha_5 \leq -388$,
							which is true, since $\alpha_5\geq \binom{12}{6}$ and $f_6(12)=-660$.
\item $n=13$. If $\alpha_5\geq \binom{12}{5}+\binom{10}{4}+\binom{6}{3}+\binom{3}{2}=1025$ then $\alpha_6-2\alpha_5\leq -880=f_6(13)$ and
              we are done. Hence, we may assume $\alpha_5\leq 1024$. Since $\alpha_3\leq \binom{13}{3}$, from \eqref{sisi} it follows that
							$$3\alpha_4-4\alpha_3+5\alpha_2 \leq 1324.$$
							Thus, in order to show that $\beta_6^7\leq \binom{11}{6}=462$, it suffices to show that
							$\alpha_6-2\alpha_5 \leq -791$. This condition holds form $\alpha_5\geq \binom{12}{5}+\binom{9}{4}+\binom{3}{3}+\binom{2}{2}+\binom{1}{1}=921$,
							hence we may assume $\alpha_5\leq 920$. From \eqref{sisi} it follows that it is enough to show $\alpha_6-2\alpha_5\leq -687$, a condition
							which is hold for $\alpha_5\geq \binom{12}{5}+\binom{6}{4}+\binom{3}{3}+\binom{2}{2}=809$. Assume $\alpha_5\leq 808$.
							From \eqref{sisi} it follows that it is enough to show $\alpha_6-2\alpha_5\leq -575$, a condition which is satisfied for 
							$\alpha_5\geq \binom{11}{5}+\binom{8}{4}+\binom{6}{3}+\binom{2}{2}+\binom{1}{1}=554$. If $\alpha_5\leq 553$, from \eqref{sisi}
							it follows that it is enough to show that $\alpha_6-2\alpha_5\leq -320$, which is true as $\alpha_5\geq \binom{11}{5}$.							
\item $n=12$. If $\alpha_5\geq \binom{11}{5}+\binom{10}{4}+\binom{6}{3}+\binom{3}{2}=695$ then $\alpha_6-2\alpha_5\leq -660=f_6(12)$ and we are done.
              Assume $\alpha_5\leq 694$. Since $\alpha_3\leq \binom{12}{3}$, from \eqref{sisi} it follows that  
							$$3\alpha_4-4\alpha_3+5\alpha_2 \leq 928.$$
							Thus, in order to show that $\beta_6^7\leq \binom{10}{6}=210$, it suffices to show that
							$\alpha_6-2\alpha_5\leq -653$. This holds for $\alpha_5\geq \binom{11}{5}+\binom{10}{4}+\binom{5}{3}+\binom{3}{2}+\binom{2}{1}=687$. 
							Assume $\alpha_5\leq 686$. From \eqref{sisi} it follows that it is enough to prove $\alpha_6-2\alpha_5\leq -645$.
							This holds for $\alpha_5\geq \binom{11}{5}+\binom{10}{4}+\binom{4}{3}+\binom{3}{2}+\binom{2}{1}=681$.
							Assume $\alpha_5\leq 680$. We have to show that $\alpha_6-2\alpha_5\leq -639$, which is true for 
							$\alpha_5\geq \binom{11}{5}+\binom{10}{4}+\binom{4}{3}+\binom{2}{2}=678$. Assume $\alpha_5\leq 677$.
							We have to show that $\alpha_6-2\alpha_5\leq -636$, which is true for $\alpha_5\geq \binom{11}{5}+\binom{10}{4}+
							\binom{3}{3}+\binom{2}{2}+\binom{1}{1}=675$. Assume $\alpha_5\leq 674$. We have to show that $\alpha_6-2\alpha_5\leq -633$,
							which is true for $\alpha_5=674$. Assume $\alpha_5\leq 673$. We have to show that $\alpha_6-2\alpha_5\leq -632$,
							which is true for $\alpha_5=673$. 
							
							Assume $\alpha_5\leq 672$. We have to show that $\alpha_6-2\alpha_5\leq -631$.
							Assume $\alpha_5=672$. If $\alpha_3 < \binom{12}{3}$ then, using \eqref{sisi}, it follows that it is enough to
							show $\alpha_6-2\alpha_5\leq -630$ which is true in this case. Now, assume $\alpha_3=\binom{12}{3}$. Since $\alpha_5=672$
							and $\beta_5^8\geq 0$ it follows that $450\leq \alpha_4\leq 479$. Since $\alpha_4\leq 479$ and $\alpha_6-2\alpha_5\leq -630$
							it follows that $\beta_6^7\leq 192 < 210 = \binom{10}{6}$ and thus we are done. 
							
							Now, assume $\alpha_5\leq 671$. From \eqref{sisi} it is enough to show that $\alpha_6-2\alpha_5\leq -630$, which is true
							for $\alpha_5\geq \binom{11}{5}+\binom{9}{4}+ \binom{7}{3} + \binom{3}{2}+\binom{1}{1}=627$. So we may assume $\alpha_5\leq 626$.
							From \eqref{sisi} it is enough to show that $\alpha_6-2\alpha_5\leq -585$, which is true for 
							$\alpha_5\geq \binom{11}{5}+\binom{8}{4}+\binom{7}{3}+\binom{2}{2}+\binom{1}{1}=569$. So we may assume $\alpha_5\leq 568$.
							From \eqref{sisi} it is enough to show that $\alpha_6-2\alpha_5\leq -527$, which is true for
							$\alpha_5\geq \binom{11}{5}+\binom{7}{4}+\binom{5}{3}+\binom{2}{2}=508$. So me may assume $\alpha_5\leq 507$.
							From \eqref{sisi} it is enough to show that $\alpha_6-2\alpha_5\leq -466$, which is true for
							$\alpha_5\geq \binom{11}{5}+\binom{4}{4}+\binom{3}{3}=464$. So me may assume $\alpha_5\leq 463$.
							From \eqref{sisi} it is enough to show that $\alpha_6-2\alpha_5\leq -422$, which is true for
							$\alpha_5\geq \binom{10}{5}+\binom{9}{4}+\binom{3}{3} = 379$. So me may assume $\alpha_5\leq 378$.
							From \eqref{sisi} it is enough to show that $\alpha_6-2\alpha_5\leq -357$, which is true for
							$\alpha_5\geq \binom{10}{5}+\binom{7}{4}+\binom{4}{3}+\binom{2}{2}+\binom{1}{1}=293$. So me may assume $\alpha_5\leq 292$.
							From \eqref{sisi} it is enough to show that $\alpha_6-2\alpha_5\leq -271$, which is true, as $\alpha_5\geq \binom{10}{5}$.
\item $n=11$. If $\alpha_5\geq \binom{10}{5}+\binom{9}{4}+\binom{7}{3}+\binom{3}{2}+\binom{1}{1}=417$ then
              $\alpha_6-2\alpha_5\leq -462=f_6(11)$, so we may assume $\alpha_5\leq 416$. From \eqref{sisi} it follows that
							$$3\alpha_4-4\alpha_3+5\alpha_2 \leq 416+\binom{11}{3}+14=595.$$
							Thus, in order to show that $\beta_6^7\leq \binom{9}{6}=84$, it suffices to show that
							$\alpha_6-2\alpha_5\leq -452$. This is true for $\alpha_5\geq \binom{10}{5}+\binom{9}{4}+\binom{6}{3}
							+\binom{3}{2}+\binom{1}{1}=402$, so we may assume $\alpha_5\leq 401$. From \eqref{sisi} it follows that
							it is enough to show that $\alpha_6-2\alpha_5\leq -437$. This is true for 
							$\alpha_5\geq \binom{10}{5}+\binom{9}{4}+\binom{5}{3}+\binom{2}{2}=389$, so we may assume $\alpha_5\leq 388$.
							From \eqref{sisi} it follows that
							it is enough to show that $\alpha_6-2\alpha_5\leq -425$. This is true for 
							$\alpha_5\geq \binom{10}{5}+\binom{9}{4}+\binom{3}{3}+\binom{2}{2}+\binom{1}{1}=381$,
							so we may assume $\alpha_5\leq 380$. From \eqref{sisi} it follows that
							it is enough to show that $\alpha_6-2\alpha_5\leq -417$. This is true for 
							$\alpha_5\geq \binom{10}{5}+\binom{8}{4}+\binom{7}{3}+\binom{2}{2}+\binom{1}{1}=358$,
							so we may assume $\alpha_5\leq 357$. From \eqref{sisi} it follows that
							it is enough to show that $\alpha_6-2\alpha_5\leq -394$. This is true for 
							$\alpha_5\geq \binom{10}{5}+\binom{8}{4}+\binom{5}{3}+\binom{2}{2}=333$,
							so we may assume $\alpha_5\leq 332$. From \eqref{sisi} it follows that
							it is enough to show that $\alpha_6-2\alpha_5\leq -369$. This is true for
							$\alpha_5\geq \binom{10}{5}+\binom{7}{4}+\binom{6}{3}+\binom{2}{2}=308$,
							so we may assume $\alpha_5\geq 307$. From \eqref{sisi} it follows that
							it is enough to show that $\alpha_6-2\alpha_5\leq -344$. This is true for
							$\alpha_5\geq \binom{10}{5}+\binom{7}{4}+\binom{3}{3}=288$, so we may
							assume $\alpha_5\geq 287$. From \eqref{sisi} it follows that
							it is enough to show that $\alpha_6-2\alpha_5\leq -324$. This is true for
							$\alpha_5\geq \binom{10}{5}+\binom{6}{4}+\binom{3}{3}+\binom{2}{2}+\binom{1}{1}=270$,
							so we may assume $\alpha_5\geq 269$. From \eqref{sisi} it follows that
							it is enough to show that $\alpha_6-2\alpha_5\leq -306$. This is true for
							$\alpha_5\geq \binom{10}{5}+\binom{5}{4}+\binom{3}{3}+\binom{2}{2}=259$,
							so we may assume $\alpha_5\geq 258$. From \eqref{sisi} it follows that
							it is enough to show that $\alpha_6-2\alpha_5\leq -295$. This is true for
							$\alpha_5\geq \binom{10}{5}+\binom{4}{4}=253$, so we may assume $\alpha_5\leq 252$. 
							From \eqref{sisi} it follows that
							it is enough to show that $\alpha_6-2\alpha_5\leq -289$. This is true for
							$\alpha_5\geq \binom{9}{5}+\binom{8}{4}+\binom{7}{3}+\binom{2}{2}=232$,
							so we may assume $\alpha_5\leq 231$. From \eqref{sisi} it follows that
							it is enough to show that $\alpha_6-2\alpha_5\leq -268$. This is true for
							$\alpha_5\geq \binom{9}{5}+\binom{8}{4}+\binom{4}{3}+\binom{3}{2}+\binom{1}{1}=204$,
							so we may assume $\alpha_5\leq 203$. From \eqref{sisi} it follows that
							it is enough to show that $\alpha_6-2\alpha_5\leq -240$. This is true for
							$\alpha_5\geq \binom{9}{5}+\binom{7}{4}+\binom{5}{3}+\binom{3}{2}+\binom{2}{1}=176$,
							so me may assume $\alpha_5\leq 175$. From \eqref{sisi} it follows that
							it is enough to show that $\alpha_6-2\alpha_5\leq -212$. This is true for
							$\alpha_5\geq \binom{9}{5}+\binom{6}{4}+\binom{5}{3}+\binom{3}{2}=154$,
							so we may assume $\alpha_5\leq 153$. From \eqref{sisi} it follows that
							it is enough to show that $\alpha_6-2\alpha_5\leq -190$.  This is true for
							$\alpha_5\geq \binom{9}{5}+\binom{5}{4}+\binom{4}{3}+\binom{3}{2}+\binom{1}{1}=139$,
							so we may assume $\alpha_5\leq 138$. From \eqref{sisi} it follows that
							it is enough to show that $\alpha_6-2\alpha_5\leq -175$. This is true, since
							$\alpha_5\geq \binom{9}{5}+\binom{5}{4}$.
\item $n=10$. If $\alpha_5\geq \binom{9}{5}+\binom{8}{4}+\binom{7}{3}+\binom{3}{2}+\binom{1}{1}=235$ then
              $\alpha_6-2\alpha_5\leq -294=f_6(10)$, so we may assume $\alpha_5\leq 234$. From \eqref{sisi} it follows that
							$$3\alpha_4-4\alpha_3+5\alpha_2 \leq 234+\binom{10}{3}+14=368.$$
							Thus, in order to show that $\beta_6^7\leq \binom{8}{6}=28$, it suffices to show that
							$\alpha_6-2\alpha_5\leq -287$. This is true for $\alpha_5\geq \binom{9}{5}+\binom{8}{4}+\binom{7}{3}=231$,
							so we may assume $\alpha_5\leq 230$. From \eqref{sisi} it follows that
							it is enough to show that $\alpha_6-2\alpha_5\leq -283$. This is true for
							$\alpha_5\geq \binom{9}{5}+\binom{8}{4}+\binom{6}{3}+\binom{3}{2}+\binom{1}{1}=220$,
							so we may assume $\alpha_5\leq 219$. From \eqref{sisi} it follows that
							it is enough to show that $\alpha_6-2\alpha_5\leq -272$. This is true for
							$\alpha_5\geq \binom{9}{5}+\binom{8}{4}+\binom{5}{3}+\binom{3}{2}=209$,
							so we may assume $\alpha_5\leq 208$.  From \eqref{sisi} it follows that
							it is enough to show that $\alpha_6-2\alpha_5\leq -261$. This is true for
							$\alpha_5\geq \binom{9}{5}+\binom{8}{4}+\binom{4}{3}+\binom{2}{2}=201$,
							so we may assume $\alpha_5\leq 200$. From \eqref{sisi} it follows that
							it is enough to show that $\alpha_6-2\alpha_5\leq -253$. This is true for
							$\alpha_5\geq \binom{9}{5}+\binom{8}{4}+\binom{4}{4}=197$, 
							so we may assume $\alpha_5\leq 196$. From \eqref{sisi} it follows that
							it is enough to show that $\alpha_6-2\alpha_5\leq -249$. This is true for
							$\alpha_5\geq \binom{9}{5}+\binom{7}{4}+\binom{6}{3}+\binom{3}{2}+\binom{1}{1}=185$,
							so we may assume $\alpha_5\leq 184$. From \eqref{sisi} it follows that
							it is enough to show that $\alpha_6-2\alpha_5\leq -237$. This is true for
							$\alpha_5\geq \binom{9}{5}+\binom{7}{4}+\binom{5}{3}+\binom{3}{2}=174$,
							so we may assume $\alpha_5\leq 173$. From \eqref{sisi} it follows that
							it is enough to show that $\alpha_6-2\alpha_5\leq -226$. This is true for
							$\alpha_5\geq \binom{9}{5}+\binom{7}{4}+\binom{4}{3}+\binom{2}{2}+\binom{1}{1}= 167$,
							so we may assume $\alpha_5\leq 166$. From \eqref{sisi} it follows that
							it is enough to show that $\alpha_6-2\alpha_5\leq -219$. This is true for
							$\alpha_5\geq \binom{9}{5}+\binom{7}{4}+\binom{3}{3}=162$,
							so we may assume $\alpha_5\leq 161$. From \eqref{sisi} it follows that
							it is enough to show that $\alpha_6-2\alpha_5\leq -214$. This is true for
							$\alpha_5\geq \binom{9}{5}+\binom{6}{4}+\binom{5}{3}+\binom{3}{2}+\binom{1}{1}=155$,
							so we may assume $\alpha_5\leq 154$. From \eqref{sisi} it follows that
							it is enough to show that $\alpha_6-2\alpha_5\leq -207$. This is true for
							$\alpha_5\geq \binom{9}{5}+\binom{6}{4}+\binom{4}{3}+\binom{3}{2}+\binom{2}{1}=150$,
							so we may assume $\alpha_5\leq 149$. From \eqref{sisi} it follows that
							it is enough to show that $\alpha_6-2\alpha_5\leq -202$. This is true for
							$\alpha_5\geq \binom{9}{5}+\binom{6}{4}+\binom{4}{3}+\binom{2}{2}+\binom{1}{1}=147$,
							so we may assume $\alpha_5\leq 146$. From \eqref{sisi} it follows that
							it is enough to show that $\alpha_6-2\alpha_5\leq -199$. This is true for
							$\alpha_5\geq \binom{9}{5}+\binom{6}{4}+\binom{4}{3}=145$, so we may assume $\alpha_5\geq 144$. From \eqref{sisi} 
							it follows that it is enough to show that $\alpha_6-2\alpha_5\leq -197$. 
							This is true for $\alpha_5=144$, so we may assume $\alpha_5\leq 143$ and it enough to show 
							$\alpha_6-2\alpha_5\leq -196$. Again, this relation holds for $\alpha_5=143$ so we may asume
							$\alpha_5\leq 142$. 
							
							If $\alpha_3\geq 113$ then $\alpha_2=\binom{10}{2}$. Since $\beta_3^8\geq 0$ it follows that $\alpha_3\geq 116$.
							Now, since $\beta_4^8\geq 0$ it follows that $\alpha_4\geq \binom{10}{4}-25$. But, as $\alpha_5\leq 142$ it follows
							that $\beta_5^8\leq -44<0$, a contradiction.
							
						  On the other hand, if $\alpha_3\leq 112$ then from \eqref{sisi} it follows that it is enough to show that
							$\alpha_6-2\alpha_5 \leq -188$, which is true for 
							$$\alpha_5 \geq \binom{9}{5}+\binom{5}{4}+\binom{4}{3}+\binom{2}{2}+\binom{1}{1}=137.$$
							So, we may assume $\alpha_5\leq 136$. From the fact that $\alpha_3\leq 112$ and \eqref{sisi} it follows that
							it is enough to show $\alpha_6-2\alpha_5\leq -180$, which is true for $\alpha_5\geq \binom{9}{5}+\binom{5}{4}+\binom{3}{3}+\binom{2}{2}=133$.
							So, we may assume $\alpha_5\leq 132$ and we have to prove that $\alpha_6-2\alpha_5\leq -176$, which is 
							true for $\alpha_5\geq \binom{9}{5}+\binom{4}{4}+\binom{3}{3}+\binom{2}{2}+\binom{1}{1}=130$.
							So, we may assume $\alpha_5\leq 129$ and we have to prove that $\alpha_6-2\alpha_5\leq -173$, which is true for $\alpha_5=129$.
							If $\alpha_5=128$ then we have to prove that $\alpha_6-2\alpha_5\leq -172$, which is true. 
							
							Now, assume $\alpha_5\leq 127$. If $\alpha_3\geq 106$ then $\alpha_2\geq \binom{10}{2}-1$. Since $\beta_3^8\geq 0$ it follows
							that $\alpha_3\geq 110$. As $\beta_4^8\geq 0$ it follow that $\alpha_4\geq \binom{10}{4}-40$. Since $\alpha_5\leq 127$ it follows
							$\beta_5^8<0$, a contradiction. So, we may assume $\alpha_3\leq 105$ instead of $\alpha_3\leq 112$ and it suffices to show
							that $\alpha_6-2\alpha_5\leq -164$, which is true for
							$\alpha_5\geq \binom{8}{5}+\binom{7}{4}+\binom{6}{3}+\binom{3}{2}+\binom{2}{1}=116.$
							We assume therefore $\alpha_5\leq 115$ and we have to prove that $\alpha_6-2\alpha_5\leq -152$, which is true for
							$\alpha_5\geq \binom{8}{5}+\binom{7}{4}+\binom{5}{3}+\binom{2}{2}+\binom{1}{1}=103.$
							We assume $\alpha_5\leq 102$ and we have to prove that $\alpha_6-2\alpha_5\leq -139$, which is true for
							$\alpha_5\geq \binom{8}{5}+\binom{7}{4}+\binom{3}{3}+\binom{2}{2}+\binom{1}{1}=94.$
							We assume $\alpha_5\leq 93$ and we have to prove that $\alpha_6-2\alpha_5\leq -130$, which is true for
							$\alpha_5\geq \binom{8}{5}+\binom{6}{4}+\binom{5}{3}+\binom{3}{2}+\binom{1}{1}=85.$
							We assume $\alpha_5\leq 84$ and we have to prove that $\alpha_6-2\alpha_5\leq -121$, which is true for
							$\alpha_5\geq \binom{8}{5}+\binom{6}{4}+\binom{4}{3}+\binom{3}{2}+\binom{1}{1}=79.$
							We assume $\alpha_5\leq 78$ and we have to prove that $\alpha_6-2\alpha_5\leq -115$, which is true for
							$\alpha_5\geq \binom{8}{5}+\binom{6}{4}+\binom{4}{3}=75.$
							We assume $\alpha_5\leq 74$ and we have to prove that $\alpha_6-2\alpha_5\leq -111$,
							which is true for $\alpha_5\geq \binom{8}{5}+\binom{6}{4}+\binom{3}{3}+\binom{2}{2}=73.$
							We assume $\alpha_5\leq 72$ and we have to prove that $\alpha_6-2\alpha_5\leq -109$, which is true for $\alpha_5=72$.
							We assume $\alpha_5\leq 71$ and we have to prove that $\alpha_6-2\alpha_5\leq -108$, which is true
							for $\alpha_5\geq \binom{8}{5}+\binom{5}{4}+\binom{4}{3}+\binom{3}{2}+\binom{2}{1}=70.$
							We assume $\alpha_5\leq 69$ and we have to prove that $\alpha_6-2\alpha_5\leq -106$, which is true
							for $\alpha_5 = 69$. We assume $\alpha_5\leq 68$ and we have to prove that $\alpha_6-2\alpha_5\leq -105$.							
							This is true for $\alpha_5=68$ so we may assume $\alpha_5\leq 67$. Also, we have to prove $\alpha_6-2\alpha_5\leq -104$.
							
							Again, this is true for $\alpha_5=67$. Assume $\alpha_5\leq 66$. We have to prove that $\alpha_6-2\alpha_5\leq -103$.
							If $\alpha_3\geq 100$ then $\alpha_2\geq \binom{10}{2}-2$. Since $\beta_3^8\geq 0$ it follows that $\alpha_3\geq 104$.
							Since $\beta_4^8\geq 0$ it follows that $\alpha_4\geq \binom{10}{4}-55$. It is easy to see that $\beta_5^8<0$ a contradiction.
							On the other hand, if $\alpha_3\leq 99$, from \eqref{sisi} it follows that it is enough to prove that 
							$\alpha_6-2\alpha_5\leq -97$, which is true for $\alpha_5\geq \binom{8}{5}+\binom{5}{4}+\binom{3}{3}+\binom{2}{2}=63$.
							Hece, we may assume $\alpha_5\leq 62$ and it is enough to prove that $\alpha_6-2\alpha_5\leq -93$, which is true
							since $\alpha_5\geq \binom{8}{5}+\binom{5}{4}$.							
							
\end{itemize}
Hence, the proof is complete.
\end{proof}

\begin{lema}\label{q7_8}
If $q=8$ then $\beta_7^{7} \leq \binom{n-1}{7}$.
\end{lema}

\begin{proof}
Since $\beta_7^7 = \alpha_7 - \beta_6^6$ and $\beta_6^6\geq 0$, we can assume that $\alpha_7\geq \binom{n-1}{7}+1$, otherwise there
is nothing to prove. It follows that:
\begin{equation}\label{condz}
\alpha_k \geq \binom{n-1}{k} + \binom{6}{k-1},\text{ for }2\leq k\leq 7.
\end{equation}
Also, as $\beta_k^8\geq 0$, for $2\leq k\leq 6$, it follows that
\begin{equation}\label{condyz}
\begin{split}
& \alpha_2 \geq 7(n-4),\;\alpha_3 \geq 6\alpha_2-21n+56 \geq 21n-112 \\
& \alpha_4 \geq 5\alpha_3 - 15\alpha_2 + 35n - 70 \geq 35n- 210 \\
& \alpha_5 \geq 4\alpha_4 - 10\alpha_3 + 20\alpha_2 - 35n + 56 \geq  35n - 224 \\
& \alpha_6 \geq 3\alpha_5 - 6\alpha_4 + 10 \alpha_3 - 15\alpha_2 + 21n - 28.
\end{split}
\end{equation}
Let $h_3(x)=\binom{x}{3}-\binom{x}{2}$ and $h_2(x)=\binom{x}{2}-\binom{x}{1}$. We claim that
\begin{equation}\label{kklem1}
\alpha_3-\alpha_2 \leq \binom{n}{3}-\binom{n}{2} = h_3(n).
\end{equation}
If $\alpha_2=\binom{n}{2}$ then there is nothing to prove. Otherwise, from \eqref{condz} it follows that
$\alpha_2=\binom{n-1}{2}+\binom{n_1}{1}$ for some $n-1>n_1\geq 6$. Since $\alpha_3\leq \binom{n-1}{3}+\binom{n_1}{2}$ and $n\geq 10$
 it follows that $$\alpha_3-\alpha_2 \leq h_3(n-1)+h_2(n_2) \leq h_3(n-1)+h_2(n-2) = h_3(n)-(n-3)\leq h_3(n),$$
and thus \eqref{kklem1} holds.

Let $g_k(x)=\binom{x}{k}-\binom{x}{k-1}$ for $2\leq k\leq 5$.
We claim that
\begin{equation}\label{kklem2}
\alpha_5-\alpha_4 \leq \binom{n}{5}-\binom{n}{4} = g_5(n).
\end{equation}
If $\alpha_4=\binom{n}{4}$ then there is nothing to prove. Otherwise, from \eqref{condz} it follows that
$$\alpha_4=\binom{n-1}{4}+\binom{n_3}{3}+\binom{n_2}{2}+\binom{n_1}{1},$$ for some $n-1\geq n_3\geq 6$ and $n_3>n_2>n_1\geq 0$.
Since $n\geq 10$, we have
\begin{align*}
& h_5(n)-(\alpha_5-\alpha_4) \geq h_5(n) - (h_5(n-1)+h_4(n_3)+h_3(n_2)+h_2(n_1)) \geq \\
& h_5(n) - (h_5(n-1)+h_4(n-2)+h_3(n-3)+h_2(n-4)) = n-5 \geq 0,
\end{align*}
and thus the claim \eqref{kklem2} is proved.

Now, let $f_k(x)=\binom{x}{k}-\binom{x}{k-1}$, for $2\leq k\leq 7$.

 \begin{center}
 \begin{table}[tbh]
 \begin{tabular}{|c|c|c|c|c|c|c|c|c|c|c|c|c|c|c|}
 \hline
 $x$      & 1  & 2  & 3  & 4  & 5  & 6  & 7   & 8    & 9  & 10 & 11 & 12 & 13 & 14   \\ \hline
 $f_7(x)$ & 0  & 0  & 0  & 0  & 0  & -1 & -6  & -20  & -48 & -90& -132& -132& 0& 429 \\ \hline
 $f_6(x)$ & 0  & 0  & 0  & 0  & -1 & -5 & -14 & -28  & -42 & -42& 0   & 132 & 429 &  \\ \hline
 $f_5(x)$ & 0  & 0  & 0  & -1 & -4 & -9 &  -14 & -14 & 0   & 42 & 132 & 297 &   &  \\ \hline
 $f_4(x)$ & 0  & 0  & -1 & -3 & -5 & -5 & 0   &  14  & 42  & 90 & 165 &     &   &  \\ \hline
 $f_3(x)$ & 0  & -1 & -2 & -2 & 0 & 5 &  14 & 28 &  48 & 75 & & & &  \\ \hline
 $f_2(x)$ & -1 & -1 & 0  & 2  & 5 & 9 & 14 &  20 &  27 & & & & & \\ \hline
 \end{tabular}
 \end{table}
\end{center}

From \eqref{kklem1} and \eqref{kklem2} it follows that, in order to prove that $\beta_7^{7} \leq \binom{n-1}{7}$ it
suffices to show that
\begin{equation}\label{kklem3}
\alpha_7-\alpha_6 \leq \binom{n}{7}-\binom{n}{6} = f_7(n).
\end{equation}
If $\alpha_6=\binom{n}{6}$ then there is nothing to prove, so we may assume that
$$\alpha_6=\binom{n-1}{6}+\binom{n_5}{5}+\binom{n_4}{4}+\binom{n_3}{3}+\binom{n_2}{2}+\binom{n_1}{1}$$
where $\;n-1>n_5\geq 7\text{ and }n_5>n_4>n_3>n_2>n_1\geq 0$.

If $n\geq 13$ then, using the table with the values for $f_k(x)$'s, it is easy to see that
\begin{align*}
& f_7(n)-(\alpha_7-\alpha_6)\geq f_7(n)-(f_7(n-1)+f_6(n_5)+f_5(n_4)+f_4(n_3)+f_3(n_2)+f_2(n_1)) \geq \\
& f_7(n)-(f_7(n-1)+f_6(n-2)+f_5(n-3)+f_4(n-4)+f_3(n-5)+f_2(n-6)) = n-7 \geq 6,
\end{align*}
and thus \eqref{kklem3} holds. It remains to consider the following cases:
\begin{itemize}
\item $n=12$. Since $n_6\geq 6$, from the table with the values for $f_k(x)$'s it follows that
$$\alpha_7-\alpha_6 \leq f_7(11) + f_6(10)+f_5(9)+f_4(8)+f_3(7)+f_2(1) = -137<-132=f_7(12),$$
and we are done.
\item $n=11$. If $\alpha_6\geq \binom{10}{6}+\binom{8}{5}+\binom{7}{4}=301$ then 
      $$\alpha_7-\alpha_6\leq f_7(10)+f_6(8)+f_5(7)=-132=f_7(11),$$
			and there is nothing to prove. Hence, we may assume that $\alpha_6\leq 300$. From \eqref{condyz}
			it follows that
			$$3\alpha_5 - 6\alpha_4 + 10\alpha_3 - 15\alpha_2 \leq 300  - 21\cdot 11 + 28 = 97,$$
			which is equivalent to
			\begin{equation}\label{ss1}
			3(\alpha_5-\alpha_4+\alpha_3-\alpha_2) \leq 97 + 3\alpha_4 - 7\alpha_3 +12\alpha_2.
			\end{equation}
			On the other hand, we claim that 
			\begin{equation}\label{ss2}
			3\alpha_4 - 7\alpha_3\leq 3\binom{11}{4}-7\binom{11}{3}=-165.
			\end{equation}
			Indeed, if $\alpha_3=\binom{11}{3}$ then there is nothing to prove. Assume that $\alpha_3=\binom{10}{3}+\binom{n_2}{2}+\binom{n_1}{1}$,
			where $10>n_2>n_1\geq 0$ and $n_2\geq 6$, it is easy to check that
			$$3\alpha_4 - 7\alpha_3\leq 3\left(\binom{10}{4}+\binom{9}{3}+\binom{8}{2}\right)-7\left(\binom{10}{3}+\binom{9}{2}+\binom{8}{1}\right)=-182.$$
			From \eqref{ss1} and \eqref{ss2} it follows that
			$$\alpha_5-\alpha_4+\alpha_3-\alpha_2\leq 197.$$
			Hence, in order to show that $\beta_7^7\leq \binom{10}{7}=120$, it is enough to prove that
			$\alpha_7-\alpha_6\leq -87$, which is true as $\alpha_7-\alpha_6<f_7(10)=-90$.
			
\item $n=10$. If $\alpha_6\geq \binom{9}{6}+\binom{8}{5}+\binom{6}{4}+\binom{4}{3}+\binom{2}{2}+\binom{1}{1}=161$
      then $$\alpha_7-\alpha_6\leq f_7(9)+f_6(8)+f_5(6)+f_4(4)+f_3(2)+f_2(1)=-90=f_7(10),$$
			and there is nothing to prove. Hence, we may assume that $\alpha_6\leq 160$. From \eqref{condyz}
			it follows that
			$$3\alpha_5 - 6\alpha_4 + 10\alpha_3 - 15\alpha_2 \leq 160  - 21\cdot 10 + 28 = -22 ,$$
			which is equivalent to
			\begin{equation}\label{sss1}
			3(\alpha_5-\alpha_4+\alpha_3-\alpha_2) \leq -22 + 3\alpha_4 - 7\alpha_3 +12\alpha_2.
			\end{equation}
			As in the case $n=11$, it is easy to check that $3\alpha_4 - 7\alpha_3\leq 3\binom{10}{4}-7\binom{10}{3}=-210$.
			From \eqref{sss1} we deduce therefore that $$\alpha_5-\alpha_4+\alpha_3-\alpha_2\leq 102.$$
			Hence, in order to show that $\beta_7^7\leq \binom{9}{7}=36$, it is enough to prove that
			$\alpha_7-\alpha_6\leq -76$, which is true for $\alpha_6\geq \binom{9}{6}+\binom{7}{5}+\binom{6}{4}+\binom{4}{3}+\binom{2}{2}+\binom{1}{1}=126$.
			Assume $\alpha_6\leq 125$.
			From \eqref{condyz} it follows that $$3\alpha_5 - 6\alpha_4 + 10\alpha_3 - 15\alpha_2 \leq -56.$$ As above, we deduce that
			$\alpha_5-\alpha_4+\alpha_3-\alpha_2\leq 91$ and thus it is enough to show that $\alpha_7-\alpha_6\leq -65$.
			
			If $\alpha_6\geq \binom{9}{6}+\binom{7}{5}+ \binom{4}{4}+\binom{3}{3}+\binom{2}{2}=108$ then $\alpha_7-\alpha_6\leq -65$ and there is nothing
			to prove, hence we may assume $\alpha_6\leq 107$. From \eqref{condyz} it follows that 
			$$3\alpha_5 - 6\alpha_4 + 10\alpha_3 - 15\alpha_2 \leq -75.$$
			As above, we deduce $\alpha_5-\alpha_4+\alpha_3-\alpha_2\leq 85$ thus it is enough to show that $\alpha_7-\alpha_6\leq -59$. Again, this is
			true for $\alpha_6\geq \binom{9}{6}+\binom{6}{5}+\binom{5}{4}+\binom{3}{3}+\binom{2}{2}=97$, so we may assume $\alpha_6\leq 96$.
			
			From \eqref{condyz} it follows that $$3\alpha_5 - 6\alpha_4 + 10\alpha_3 - 15\alpha_2 \leq -86.$$
			As above, we deduce $\alpha_5-\alpha_4+\alpha_3-\alpha_2\leq 81$ thus it is enough to show that $\alpha_7-\alpha_6\leq -55$,
			which is true for $\alpha_6\geq \binom{9}{6}+\binom{6}{5}+ \binom{4}{4}+\binom{3}{3}=92$. Assume $\alpha_6\leq 91$.
			From \eqref{condyz} it follows that $$3\alpha_5 - 6\alpha_4 + 10\alpha_3 - 15\alpha_2 \leq -91,$$ and we deduce
			$\alpha_5-\alpha_4+\alpha_3-\alpha_2\leq 79$ thus it is enough to show that $\alpha_7-\alpha_6\leq -53$, which is true
			as $\alpha_6\geq \binom{9}{6}+\binom{6}{5}$.			
\end{itemize}
Hence, the proof is complete.
\end{proof}

Now we are able to prove our main result:

\begin{teor}\label{main}
Let $I\subset S$ be a squarefree monomial ideal with $\hdepth(S/I)\leq 8$. Then
$$\hdepth(I)\geq \hdepth(S/I)-1.$$
\end{teor}

\begin{proof}
It follows from Lemma \ref{lem2}, Lemma \ref{q3_10}, Lemma \ref{q4_8}, Lemma \ref{q5_8}, Lemma \ref{q6_8} and Lemma \ref{q7_8}.
\end{proof}

\begin{cor}
Let $I\subset S=K[x_1,x_2,\ldots,x_{10}]$ be a squarefree monomial ideal. Then
$$\hdepth(I)\geq \hdepth(S/I)-1.$$
\end{cor}

\begin{proof}
Let $q=\hdepth(S/I)$. If $q=9$ then $I$ is principal and there is nothing to prove.
For $q\leq 8$ the conclusion follows from Theorem \ref{main}.
\end{proof}

\newpage
\section{Computer experiments and a conjecture}

Let $m\geq 1$ and $n\geq m+1$ be two integers. We consider the ideal
$$I_{n,m}:=(x_1x_2\cdots x_m)\cap (x_{m+1},x_{m+2},\ldots,x_n)\subset S=K[x_1,\ldots,x_n].$$
Using \cite[Example 3.4]{uli}, we can easily see that
$$\hdepth(I_{n,m})=\hdepth((x_{m+1},x_{m+2},\ldots,x_n)S)=m+\left\lfloor \frac{n-m+1}{2} \right\rfloor=\left\lfloor \frac{n+m+1}{2} \right\rfloor.$$
On the other hand, the computation of $\hdepth(S/I_{n,m})$ seems extremely difficult.

\begin{lema}\label{liema}
With the above notations
$$\alpha_j=\alpha_j(S/I_{n,m})=\begin{cases} \binom{n}{j},& 0\leq j\leq m \\ \binom{n}{j}-\binom{n-m}{j-m},& m+1\leq j\leq n \end{cases}.$$
\end{lema}

\begin{proof}
It is clear that the number of squarefree monomials of degree $j$ in $I_{n,m}$ is equal to $0$ for $j\leq m$ and $\binom{n-m}{j-m}$ for $j\geq m+1$.
The conclusion follows.
\end{proof}

\begin{prop}\label{propoo}
With the above notations
$$\beta_k^q=\beta_k^q(S/I_{n,m})=\binom{n-q+k-1}{k}-\binom{n-q+k-1-m}{k-m}+(-1)^{k-m}\binom{q-m}{k-m},$$
for all $0\leq k\leq q\leq n$.
\end{prop}

\begin{proof}
From Lemma \ref{liema} it follows that
$$\beta_k^q=\sum_{j=0}^k (-1)^{k-j} \binom{q-j}{k-j} \alpha_j = \sum_{j=0}^k (-1)^{k-j} \binom{q-j}{k-j} \left(\binom{n}{j}-\binom{n-m}{j-m}+\delta_{j}(m)\right),$$
where $\delta_{j}(m)=\begin{cases} 1,& j=m \\ 0,& j\neq m \end{cases}$. Thus, the conclusion follows by applying \eqref{combi}.
\end{proof}

\begin{cor}\label{coroo}
With the above notations
\begin{align*}
\hdepth(S/I_{n,m})=\max\{q\;:& \;\binom{n-q+k-1}{k}-\binom{n-q+k-1-m}{k-m} \geq \binom{q-m}{k-m},\\
&\text{ for all }m\leq k\leq q\leq n\text{ with }k-m\text{ odd}\}.
\end{align*}
\end{cor}

\begin{proof}
It follows from Proposition \ref{propoo} and \eqref{hdep}. 
\end{proof}

Using Corollary \ref{coroo}, we are able to calculate $\hdepth(S/I_{n,m})$ with the help of a computer. We denote
$$d(n,m)=\hdepth(S/I_{n,m})-\hdepth(I_{n,m})\text{ and }q(n,m)=\hdepth(S/I_{n,m}).$$
We have the following list of minimal, with respect to $n$ and $q(n,m)$, examples:

\begin{center}
\begin{table}[tbh]
\begin{tabular}{|c|c|c|}
\hline
$(n,m)$ & $d(n,m)$ & $q(n,m)$ \\ \hline
  (6,2) & 0        &  4       \\ \hline
  (10,2)& 1        &  7       \\ \hline
  (15,3)& 2        &  11      \\ \hline
  (20,4)& 3        &  15      \\ \hline
(25,5)  & 4        &  19      \\ \hline
(30,6)  & 5        &  23      \\ \hline
(35,5)  & 6        &  26      \\ \hline
(40,8)  & 7        &  31      \\ \hline
(45,9)  & 8        &  35      \\ \hline
(51,7)  & 9        &  38      \\ \hline
(55,11) & 10       &  43      \\ \hline
(106,20) & 20      &  83      \\ \hline
(139,17) & 30      &  108     \\ \hline
(161,19) & 40      &  130     \\ \hline
(183,21) & 50      &  152     \\ \hline
(350,7)  & 100     &  279     \\ \hline
\end{tabular}
\end{table}
\end{center}

For any integer $d\geq 0$, let $n(d)$ and $m(d)$ the smallest integers with the property that 
$$\hdepth(S/I_{n(d),m(d)})-\hdepth(I_{n(d),m(d)})=d.$$
Also, let $q(d)=\hdepth(S/I_{n(d),m(d)})$. According to the above table, we have 
for instance $n(0)=6$, $m(0)=2$, $q(0)=4$, $n(1)=10$, $m(1)=2$ and $q(1)=7$ etc.

We believe that the ideals $I_{n,m}$ give the smallest examples with respect to $n$ and to $q=\hdepth(S/I)$
for a given difference $d=\hdepth(S/I)-\hdepth(I)$, where $I\subset S$ is a squarefree monomial ideal.
More precisely, we propose the following conjecture:

\begin{conj}\label{conju}
Let $d\geq 0$ be an integer and let $I\subset S=K[x_1,\ldots,x_n]$ be a squarefree monomial ideal with $\qdepth(S/I)=q$.
If $q<q(d)$ or $n<n(d)$ then $$\hdepth(S/I)\geq \hdepth(I)-d+1.$$
\end{conj}

\begin{obs}\rm
For instance, since $n(1)=10$ and $q(1)=7$, Conjecture \ref{conju} says that if 
$\qdepth(S/I)\leq 6$ or $n\leq 9$ then $\qdepth(I)\geq \qdepth(S/I)$, a result which was proved in 
\cite{bordi} and \cite{bordi2}. Also, since $n(2)=15$ and $q(2)=11$, Conjecture \ref{conju} says that if 
$\qdepth(S/I)\leq 10$ or $n\leq 14$ then $\qdepth(I)\geq \qdepth(S/I)-1$. However, in our paper we were able to
prove only for $\hdepth(S/I)\leq 8$, that $\qdepth(I)\geq \qdepth(S/I)-1$.
\end{obs}

\subsection*{Data availability}

Data sharing not applicable to this article as no data sets were generated or analyzed
during the current study.

\subsection*{Conflict of interest}

The authors have no relevant financial or non-financial interests to disclose.


\end{document}